\documentclass[review,onefignum,onetabnum]{siamart171218}
\usepackage{multirow}
\usepackage{amsmath}
\usepackage{color}
\usepackage{amssymb}
\usepackage{mathrsfs}
\usepackage{ulem}

\newcommand\vvec{\mathbf{v}}

\newcommand\x{\mathbf{x}}
\newcommand\y{\mathbf{y}}
\newcommand\W{\mathbf{W}}
\newcommand\I{\mathbf{I}}
\newcommand\rvec{\mathbf{r}}
\newcommand\Hvec{\mathbf{H}}

\newtheorem{assumption}{Assumption}[section]

\usepackage{lipsum}
\usepackage{amsfonts}
\usepackage{graphicx}
\usepackage{epstopdf}
\usepackage{algorithmic}
\newsiamremark{remark}{Remark}
\newsiamremark{hypothesis}{Hypothesis}
\crefname{hypothesis}{Hypothesis}{Hypotheses}
\newsiamthm{claim}{Claim}

\definecolor{midbluepurple}{RGB}{134,0,102}

\newcommand{\TheTitle}{Derivative-Free
saddle-search algorithm}
\newcommand{\TheAuthors}{Q. Du, B. Shi, L. Zhang, X. Zheng}

\headers{\TheTitle}{\TheAuthors}
\title{{\TheTitle}\thanks{The work was partially supported by NSF DMS-2309245 
(for QD and BS); the National Key Research and Development Program of China  2024YFA0919500, and National Natural Science Foundation of China No. 12225102, T2321001, 12288101, and 12426653 (for LZ); the Natural Science Foundation of Shandong Province
(ZR2025QB01), the National Natural Science Foundation of China (12301555), the National Key R\&D Program of China (2023YFA1008903), and the Taishan Scholars Program of Shandong Province (tsqn202306083) (for XZ).  Author names are sorted alphabetically.
}}

\author{
Qiang Du\thanks{Department of Applied Physics and Applied Mathematics, and Data Science Institute, Columbia University, New York, NY 10027, USA (\email{qd2125@columbia.edu}).}
\and Baoming Shi\thanks{Corresponding author. Department of Applied Physics and Applied Mathematics, Columbia University, New York, NY 10027, USA (\email{bs3705@columbia.edu}).}
\and Lei Zhang\thanks{Beijing International Center for Mathematical Research, Center for Quantitative Biology, Center for Machine Learning Research, Peking University, Beijing 100871, China (\email{zhangl@math.pku.edu.cn}).}
\and Xiangcheng Zheng\thanks{School of Mathematics, State Key Laboratory of Cryptography and Digital Economy Security, Shandong University, Jinan, 250100, China  (\email{xzheng@sdu.edu.cn}).}
}
\ifpdf
\hypersetup{
  pdftitle={\TheTitle},
  pdfauthor={\TheAuthors}
}
\fi
\renewcommand{\TheTitle}{A Derivative-Free
Saddle-search Algorithm With Linear Convergence Rate}

\begin{document}
\maketitle

\begin{abstract}
We propose a derivative-free saddle-search algorithm designed to locate transition states using only function evaluations. The algorithm employs a nested architecture consisting of an inner eigenvector search and an outer saddle-point search. Through rigorous numerical analysis, we prove the almost sure convergence of the inner step under suitable assumptions. Furthermore, we establish the convergence of the outer search using a decaying step size, while demonstrating linear convergence under constant step size and boundedness conditions.  Numerical experiments are provided to validate our theoretical results and demonstrate the algorithm's practical applicability.
\end{abstract}

\begin{keywords}
saddle point, derivative-free algorithms, almost sure convergence
\end{keywords}

\begin{AMS}
	65C20, 65K10, 65L20, 37N30
\end{AMS}

\section{Introduction}
\label{sec:intro}
A recurring theme across many problems in the natural sciences and engineering is that the energy (free energy or objective functions) defined over a system’s physical variables exhibits a complex landscape characterized by numerous local minima and saddle points. Examples include the phase transformations \cite{gramsbergen1986landau,lifshitz2007soft}, chemical reactions \cite{vanden2010transition}, protein folding \cite{onuchic1997theory,wolynes1995navigating}, and the loss landscapes of deep neural networks \cite{goodfellow2016deep,lecun2015deep}. While the local minima correspond to stable states, the saddle points represent transition or excited states, playing a crucial role in determining transition pathways between minima \cite{zhang2016recent} and revealing the hierarchical structure of the energy landscape \cite{yin2020construction}. These features have attracted significant attention to the computation of saddle points over the past decades.

Compared with the identification of stable minima, the computation of saddle points is considerably more challenging, due to the apparent loss of stability. Moreover, the unstable directions of saddle points are in general not known beforehand. Extensive numerical algorithms have been developed to compute saddle points, which can be generally divided into two classes: path-finding methods \cite{weinan2002string,weinan2007simplified} and surface-walking methods \cite{zhang2016recent}. We focus on the latter type, which often employs first- or second-order derivatives to evolve an iteration on the energy surface toward saddle points.
Examples of this type of method include the gentlest ascent dynamics \cite{gad,quapp2014locating, liu2023constrained} and the high-index saddle dynamics \cite{2019High,su2025improved}, which require evaluating gradients (and Hessians) of the objective function.  
The dimer method and the shrinking dimer dynamics \cite{henkelman1999dimer,zhangdu2012,zhang2016optimization}
adopted a two-point gradient difference to approximate the projected Hessian in a one-dimensional space (determined by the {\it dimer} direction). 
While this avoids the explicit use of the Hessian, the gradient of the energy is still needed. Thus, due to the dependence on the derivative information, these methods are not directly applicable to problems where such information might be unavailable or are difficult to access, e.g., distributed learning \cite{li2021distributed}, the free energy surface in molecular simulations \cite{mezei1986free}, which is not an explicit function of atomic coordinates but rather a statistical average over microscopic states, hyperparameter tuning in machine learning \cite{snoek2012practical}, and empirical objective function from science problems \cite{deming1978review, gray2004optimizing}. We refer to \cite{audet2017introduction, rios2013derivative} for other practical problems where first- and higher-order derivatives are unavailable or difficult to access. 
Therefore, the development of derivative-free algorithms for computing saddle points is of great importance.

Zeroth-order algorithms, which rely solely on function evaluations, are applicable to the aforementioned problems and are also referred to as derivative-free or black-box methods. For example, the (stochastic) zeroth-order algorithm has been widely developed and used in optimization problems \cite{conn2009introduction,larson2019derivative}. However, very few works have extended zeroth-order methods to the saddle-point computation. Compared with zeroth-order optimization algorithms, zeroth-order saddle-search algorithms involve additional challenges for both algorithm design and convergence analysis. From the perspective of algorithm design, a zeroth-order saddle-search algorithm should also account for the update of unstable
directions (i.e., eigenvectors of the Hessian corresponding to the negative eigenvalues) that are unknowns associated with the iterates in the configuration space. This means the second-order information needs to be reconstructed from the zeroth-order information, which may lead to the curse of dimensionality \cite{jamieson2012query, nesterov2017random}. From the perspective of convergence analysis, we need to deal with the iteration of both the configuration and the unstable subspace in the saddle-search algorithms, and the objective function containing saddle points is necessarily nonconvex, which may pose challenges for proving convergence. The main contribution of this paper is to derive a zeroth-order saddle-search algorithm with convergence analysis, which is applicable to many practical problems. Notably, both our theoretical and numerical results show that the iteration points of the proposed zeroth-order algorithm with constant step size approach the saddle point with a linear convergence rate, which is comparable to that of traditional (discretized) saddle dynamics, while requiring only function evaluations. 

The paper is organized as follows. In \Cref{sec: Saddle point and saddle dynamics}, we review the concept of saddle points and the saddle dynamics. In \Cref{sec: Zeroth-order search}, we first present the zeroth-order eigenvector-search algorithm, followed by the zeroth-order saddle-search algorithm. We establish their convergence analysis in \Cref{sec: Convergence of Algorithm}. \Cref{sec: numerical results} presents numerical experiments on several benchmarks and practical problems to substantiate the theoretical results. We finally present our conclusion and discussion in \Cref{sec: conclusion}. 

\section{Saddle point and saddle dynamics}\label{sec: Saddle point and saddle dynamics}
Consider a system characterized by a smooth energy function
$
f(\x):\mathbb{R}^d\rightarrow \mathbb{R}
$
with $d\geqslant 2$, which we assume that $f$ is $C^6$ continuous. A critical point $\x^*$ is defined as a configuration where the gradient vanishes, i.e., $\nabla f(\x^*)=0$. The local stability of $\x^*$ can be identified from the Hessian matrix $\nabla^2 f(\x^*)$. If $\nabla^2 f(\x^*)$ is positive definite, then $\x^*$ is a local minimum \cite{nocedal1999numerical}; If $\nabla^2 f(\x^*)$ has  
$k\leqslant d-1$ negative eigenvalues with the corresponding $k$ eigenvectors, denoted by $\{\vvec_1^*,\cdots,\vvec_k^* \big \}$, $\x^*$ is classified as an index-$k$ saddle point, and $\{\vvec_1^*,\cdots,\vvec_k^* \big \}$ referred to as ``unstable eigenvectors" or ``unstable directions" in this paper. 

Here, we consider a particular version of the saddle dynamics for searching an index-$k$ saddle point $\x^*$, defined as 
\begin{equation}
\dot{\x}=- \left(\I-2\sum_{i=1}^k {\vvec}_i{\vvec}_i^\top\right)\nabla f(\x), 
\label{eq: SD}
\end{equation}
where $\{\vvec_i\}_{i=1}^k,1\leqslant k\leqslant d-1$ is the normalized eigen-basis corresponding to the $k$ smallest eigenvalues of $\nabla^2f(\x)$, and $\I$ is the identity operator. Unlike gradient flow, which evolves along the negative gradient and causes the objective function to decrease, the right-hand side of the saddle-point dynamics applies Householder transformations to the negative gradient along the unstable directions. As a result, energy increases along the unstable directions while decreasing along the stable directions, enabling the dynamics to locate saddle points. The eigen-basis $\{\vvec_i\}_{i=1}^k$ in \eqref{eq: SD} can be given by an eigen-solver of
$\nabla^2 f(\x)$ which generates the unstable directions.

Using the simplest time discretization of
the saddle dynamics \eqref{eq: SD}, we get the following discrete dynamics: given $\x(0)$,
\begin{equation}\label{eq: saddle algorithm}
\begin{aligned}
    &\x(n+1)=\x(n)-\alpha(n)\left(\I-2\sum_{i=1}^k\vvec_i(n)\vvec_i(n)^\top\right)\nabla f(\x(n)), \; n\geqslant 0,\\
    & \text{where } \{\vvec_i({n})\}_{i=1}^k=\text{EigenSolver}(
    \nabla^2 f(\x({n}))), 
    \text{ and  $\alpha(n)$ is the step size.}
\end{aligned}
\end{equation}

\section{{Derivative-free}/Zeroth-order saddle search}\label{sec: Zeroth-order search}
The iteration \eqref{eq: saddle algorithm} requires the evaluation of the gradient and Hessian of the objective function at each step, which might be infeasible in practice or computationally expensive \cite{audet2017introduction, rios2013derivative}. We consider a derivative-free saddle-search algorithm that requires only zeroth-order function evaluations by employing stochastic difference approximations.

\subsection{Zeroth-order approximations of the derivatives}\label{sec: Zeroth-order approximation of eigenvector}

 A natural approach to derive a zeroth-order saddle-search algorithm is to substitute $\nabla f(x)$ and $\nabla^2 f(\x)$ in 
 \eqref{eq: saddle algorithm} with zeroth-order estimators $F(\x,\rvec,l)$ and $\Hvec(\x,\rvec,l).$ Here, we use a derivative-free formulation based on a two-point random zeroth-order gradient estimator,  which has been widely used in the zeroth-order (derivative-free) optimization methods \cite{nesterov2017random,liu2018zeroth,balasubramanian2018zeroth}. It is defined by
\begin{equation}
F(\x,\mathbf{r},l)=\frac{f(\x+l\mathbf{r})-f(\x-l\mathbf{r})}{2l}\mathbf{r}, \mathbf{r}\sim \mathcal{N}(\mathbf{0},\I),
\label{eq: zeroth order gradient}
\end{equation}
where $1\gg l > 0$ is a difference length, similar to the dimer length \cite{henkelman1999dimer, zhangdu2012}, $\rvec$ are i.i.d. random directions drawn from the standard normal distribution or a uniform distribution over a unit sphere (in this case, we need to multiply $F$ by $d$ \cite{shamir2017optimal,gao2018information}). For $l>0$, although the zeroth-order gradient estimator introduces bias to the true gradient, it remains unbiased
to the gradient of a so-called Gaussian smoothing function with parameter $l$, which is defined by
\begin{equation}
f_l(\x)=\mathbb{E}_{\mathbf{r}\sim \mathcal{N}(\mathbf{0},\I)}(f(\x+l\mathbf{r})).
\label{eq:smoothed f}
\end{equation}
From Stein’s Lemma, $\nabla f_l(\x)=\mathbb{E}_{\mathbf{r}\sim \mathcal{N}(\mathbf{0},\I)}[\nabla f(\x+l\mathbf{r})]=\frac{1}{l}\mathbb{E}_{\mathbf{r}\sim \mathcal{N}(\mathbf{0},\I)}[ f(\x+l\mathbf{r})\rvec]$, and \eqref{eq: zeroth order gradient} is actually a two-point sampling, thus, it is unbiased gradient estimator of $f_l(\x)$. We note that for the saddle-point search, dimer methods \cite{henkelman1999dimer,zhangdu2012,zhang2016optimization} adopted
 gradient differences to approximate the projected Hessian in the dimer direction. However, the dimer directions are not randomly sampled, and the methods are still gradient-based. We can also use \eqref{eq:smoothed f}
to get a derivative-free approximation of the Hessian, 
$\nabla^2f_{l}(\x)=\frac{1}{l^2}\mathbb{E}_{\mathbf{r}\sim \mathcal{N}(\mathbf{0},\I)}[(\rvec\rvec^\top-\I)f(\x+l\rvec)]=\frac{1}{l^2}\mathbb{E}_{\mathbf{r}\sim \mathcal{N}(\mathbf{0},\I)}[(\rvec\rvec^\top-\I)(f(\x+l\rvec)-f(\x))]$. Consequently, the discrete sampling 
\begin{equation}
\Hvec(\x,\rvec,l)=\frac{f(\x+l\rvec)+f(\x-l\rvec)-2f(\x)}{2l^2}(\rvec\rvec^\top-\I) , \mathbf{r}\sim \mathcal{N}(\mathbf{0},\I),
\label{eq: Hessian estimator}
\end{equation}
is the unbiased zeroth-order Hessian estimator of $f_{l}(\x)$.

\subsection{A zeroth-order/derivative-free saddle search algorithm}

From \eqref{eq: saddle algorithm}, the saddle-search algorithm consists of two 
{components}:
\begin{itemize}
\item
The inner iteration, which is the eigenvector-search algorithm that identifies the unstable directions at each iteration point, see \Cref{algorithm vector}; and 
\item The outer iteration, which is the saddle-search algorithm that updates the iterate by reflecting the zeroth-order gradient estimator along the unstable directions obtained from the inner step, see \Cref{algorithm}.
\end{itemize}

We first consider the derivative-free version of the gradient descent applied to the Rayleigh quotient for the eigenvector-search, with $\nabla^2f(\x)$ replaced by the zeroth-order Hessian estimator $\Hvec$ defined in \eqref{eq: Hessian estimator}: for $n\geqslant 0$ and $\bar{k}=2,\cdots, k$,
\begin{equation}
\begin{aligned}
\hat{\vvec}(n+1)= & \vvec(n)-\alpha(n)(\I-\vvec(n)\vvec(n)^\top)\Hvec(\x,\rvec(n),l(n))\vvec(n),\\
    & \vvec(n+1)={\hat{\vvec}(n+1)}/{\Vert \hat{\vvec}(n+1)\Vert_2},    \end{aligned}
\label{stochastic eigenvector}
\end{equation}
\begin{equation}
\begin{aligned}
    \hat{\vvec}_{{\bar{k}}}(n+1)&=\vvec_{{\bar{k}}}
(n)-\alpha(n)(\I-V_{n,k}V_{n,k}^\top)\Hvec(\x,\rvec(n),l(n))\vvec_{{\bar{k}}}(n),\\
&\vvec_{{\bar{k}}}(n+1)={\hat{\vvec}_{{\bar{k}}}(n+1)}/{\Vert \hat{\vvec}_{{\bar{k}}}(n+1)\Vert_2},
 \end{aligned}  \label{stochastic eigenvector larger}
\end{equation}
where $V_{n,k}=[\vvec(n), \vvec_1,\ldots,\vvec_{\bar{k}-1}]$, and
the step size $\alpha(n)>0$ satisfies, unless otherwise noted, the following conventional assumption.

\begin{assumption}[Diminishing  step size]\label{assumption: step size sec 1}
$\displaystyle
\sum_{n=0}^\infty \alpha(n)=\infty$, and  $\,\displaystyle \sum_{n=0}^\infty \alpha(n)^2<\infty$.
\end{assumption}

However, in \Cref{sec: Dimension}, we will show that the variation of $\Hvec(\x,\rvec,l)$ suffers from the curse of dimensionality, which in turn slows down convergence. In \eqref{stochastic eigenvector} and \eqref{stochastic eigenvector larger}, it suffices to adopt an approximated Hessian-vector product \cite{zhangdu2012}. That is,
\begin{equation}
\label{eq:Hv}  H_\vvec(\x,\rvec,l)=\frac{F(\x+l\vvec,\rvec,l)-F(\x-l\vvec,\rvec,l)}{2l}
\end{equation}
can be used to replace $\Hvec(\x,\rvec,l)\vvec$, resulting in a lower variation. Because $F(\x,\rvec,l)$, $\Hvec(\x,\rvec,l)\vvec$, and $\Hvec_\vvec(\x,\rvec,l)$ are biased estimators of $\nabla f(\x)$ or $\nabla^2f(\x)\vvec$, we sometimes need the following diminishing difference length $l(n)$ to control the biased errors.
\begin{assumption}\label{assumption: decay difference length}
There exists a constant $L>0$ such that $l(n)\leqslant L\sqrt{\alpha(n)}$. 
\end{assumption}

Based on the discrete dynamics
\eqref{stochastic eigenvector} and \eqref{stochastic eigenvector larger} and the
zeroth-order Hessian-vector product estimator $H_\vvec(\x,\rvec,l)$, we propose the zeroth-order eigenvector-search algorithm in \Cref{algorithm vector}. With the \Cref{algorithm vector}, we present \Cref{algorithm} to implement the outer iteration for the derivative-free saddle-search, followed by its convergence analysis. Before we proceed to discuss the convergence of the algorithm, we present 
some assumptions and technical results in the rest of the section.

\begin{algorithm}[h]
\caption{Derivative-free eigenvector-search algorithm for $k$ unstable directions}
\label{algorithm vector}
\begin{algorithmic}
\STATE{Initial condition: $n\leftarrow0$, $\vvec_i,i=1,\cdots,k$}
\FOR{$n\leqslant n^{max}_{\vvec}$}
\STATE{$\rvec$ is drawn from $\mathcal{N}(\mathbf{0},\I)$}
\STATE{$\vvec_1\leftarrow\vvec_1-\alpha_\vvec(n)(\I-\vvec_1\vvec_1^\top)
H_{\vvec_1}(\x,\rvec,l(n))
$}
\STATE{$\vvec_1\leftarrow\vvec_1/\Vert \vvec_1\Vert_2$, $n\leftarrow n+1$}
\ENDFOR

\FOR{$j$ in $\{2,\cdots,k\}$}
\STATE{$n\leftarrow0$}
\STATE{$\vvec_j\leftarrow\vvec_j-(\I-\sum_{i=1}^{j-1}\vvec_i\vvec_i^\top)\vvec_j$, $\vvec_j\leftarrow\vvec_j/\Vert \vvec_j\Vert_2$}
\FOR{$n\leqslant n^{max}_{\vvec}$}
\STATE{$\rvec$ is drawn from  $\mathcal{N}(\mathbf{0},\I)$}
\STATE{$\vvec_j\leftarrow\vvec_j-\alpha_\vvec(n)(\I-\vvec_j\vvec_j^\top-\sum_{i=1}^{j-1}\vvec_i\vvec_i^\top)
H_{\vvec_j}(\x,\rvec,l(n))
$}
\STATE{$\vvec_j\leftarrow\vvec_j/\Vert \vvec_j\Vert_2$, $n\leftarrow n+1$}
\ENDFOR
\ENDFOR
\end{algorithmic}
\end{algorithm}

\begin{algorithm}
\caption{{Derivative-free} saddle-search algorithm for index-$k$ saddle points}
\label{algorithm}
\begin{algorithmic}
\STATE{Initial condition: $n\leftarrow 0$, $\x,\vvec_i,i=1,\cdots,k$}
\FOR{$n\leqslant n^{max}_\x$}
\STATE{$\rvec$ is drawn in $\mathcal{N}(\mathbf{0},\I)$}
\STATE{$\x\leftarrow \x-\alpha_\x(n)(\I-\sum_{i=1}^k2\vvec_i\vvec_i^T)
F(\x,\rvec,l(n))
$}
\STATE{Searching $k$ unstable directions $\{\vvec_i\}_{i=1}^k$ of $\nabla^2f(\x)$ with \Cref{algorithm vector}}
\STATE{$n\leftarrow n+1$}
\ENDFOR
\end{algorithmic}
\end{algorithm}

\subsection{Assumptions on the energy}
For brevity, we consider an objective function $f(\x)$ that satisfies the following regularity assumptions.
\begin{assumption}\label{assumption: regularity}
Assume $f\in C^6$, for all $\x\in \mathbb{R}^n$, we have $\Vert \nabla^i f(\x) \Vert_2\leqslant M$ for a positive constant $M<\infty$,  and $i=2,3,4,5,6.$    
\end{assumption}
\begin{remark}
    For the analysis of the saddle-search algorithm, taking the simplest quadratic function as an example, we can impose uniform boundedness of higher-order derivatives as in  \cite{zhang2023discretization,zhang2022error}. In many applications, there exist various objective functions satisfying \Cref{assumption: regularity}, such as the Minyaev-Quapp energy function \cite{minyaev1997internal} and the Eckhardt energy function \cite{eckhardt1988irregular}.
\end{remark}

\begin{assumption}\label{assumption: Hessian}
There exist $\delta>0$ and 
{$\mu\in (0, M)$, such that } $\forall \x\in \{\x,\Vert \x-\x^* \Vert_2\leqslant \delta\}$,  the eigenvalues of $\nabla^2 f(\x)$ satisfy
\begin{equation}
    -M\leqslant\lambda_1\leqslant\cdots\leqslant\lambda_k< -\mu < 0< \mu<\lambda_{k+1} \leqslant \cdots \leqslant \lambda_d \leqslant M.    \label{eq:eigenvalues}
\end{equation}
\end{assumption}

\begin{remark}
    The eigenvalue assumption holds for every strict saddle point $\x^*$. From \Cref{assumption: regularity}, the $L_2$-norm of the Hessian is uniformly bounded by $M$. 
Moreover, from \Cref{assumption: regularity}, for any $\vvec$ that satisfies $\Vert\vvec\Vert_2=1$, we have 
$$
\| (\nabla^2f(\x)-\nabla^2f(\mathbf{y})) \vvec\|_2=\left\|\int_0^1\nabla^3f(\mathbf{y}+t(\x-\mathbf{y}))[\x-\mathbf{y},\vvec]\mathrm{d}t\right\|_2\leqslant M\Vert\x-\mathbf{y}\Vert_2,
$$
which gives the Lipschitz condition $\Vert \nabla^2 f(\x)-\nabla^2f(\mathbf{y}) \Vert_2 \leqslant M\Vert \x-\mathbf{y} \Vert_2$,  $\forall\x,\mathbf{y}\in U$, with the constant $M$ being specified in \Cref{assumption: regularity}.
\end{remark} 

\subsection{Estimates on the approximate gradient and Hessian}
The following lemma demonstrates that the second-order moments of the zeroth-order gradient and Hessian are bounded in any bounded domain $\{\x,\Vert \x \Vert_2\leqslant C\}$.

\begin{lemma}[Almost unbiased estimator]\label{Lemma: Almost unbiased estimator}    Under \Cref{assumption: regularity}, we have $$\mathbb{E}_{\mathbf{r}}[F(\x,\mathbf{r},l)]=\nabla f_l(\x), \mathbb{E}_{\mathbf{r}}[\Hvec(\x,\rvec,l)]=\nabla^2f_l(\x),$$

$$\mathbb{E}_{\mathbf{r}}[\Vert F(\x,\mathbf{r},l)\Vert_2^2]\leqslant(2d+4)\Vert\nabla f(\x)\Vert^2_2+\frac{M^2l^4 d(d+2)(d+4)(d+6)}{18},$$
$$
\begin{aligned}
\mathbb{E}_\rvec[\Vert\Hvec(\x,\rvec,l)\Vert_2^2]
&\leqslant \frac{{M^2} d(d+2)(d^2+12d+33)}{2}\\
&\qquad+\frac{{M^2}l^4d(d+2)(d+4)(d+6)(d^2+20d+97)}{288},
\end{aligned}
$$
\begin{equation}\label{error of gradient and Hessian}
    \Vert \nabla f(\x)-\nabla f_l(\x) \Vert_2\leqslant \frac{Ml^2d}{2},
    \Vert \nabla^2 f(\x)-\nabla^2 f_l(\x) \Vert_2\leqslant \frac{Ml^2d}{2}.
\end{equation}
\end{lemma}
\begin{proof}
    This Lemma follows directly from the Taylor expansion and the higher-order moments of the normal distribution, with the proof provided in the Appendix.
\end{proof}

\begin{lemma}\label{lemma: hessian of f_l}
Under \Cref{assumption: regularity} and \Cref{assumption: Hessian}, $\nabla^2 f_l(\x)$ is Lipschitz continuous. Moreover, if $Ml^2d/2 \ll 1$, its eigenvalue satisfies that, 
for $\mu_l=\mu-Ml^2d/2>0$ and any
$\x\in \{\x,\Vert \x-\x^* \Vert_2\leqslant \delta\}$, 
\begin{equation}\label{eigenvalue of f_l}
    -M\leqslant\lambda_1 \leqslant  \cdots \leqslant \lambda_k< -\mu_l < 0< \mu_l < \lambda_{k+1} \leqslant \cdots \leqslant \lambda_d \leqslant M. 
\end{equation}
    
\end{lemma}

\begin{proof} 
To verify the Lipschitz continuity of $\nabla^2 f_l(\x)$, we compute
$$
\|\nabla^2 f_l(\mathbf{x}) - \nabla^2 f_l(\mathbf{y})\|_2 
= \left\| \mathbb{E}_{\mathbf{r}} \left[ \nabla^2 f(\mathbf{x} + l\mathbf{r}) - \nabla^2 f(\mathbf{y} + l\mathbf{r}) \right] \right\|_2  
\leqslant   M \|\mathbf{x} - \mathbf{y}\|_2 .
$$
From \Cref{assumption: regularity}, we have that $$\|\nabla^if_l(\x)\|_2=\|\mathbb{E}_\rvec[\nabla^i f(\x+l\rvec)]\|_2\leqslant \mathbb{E}_\rvec[ \Vert \nabla^i f(\x+l\rvec)\Vert_2]\leqslant M, i=2,3,4,5,6.$$
Then, \eqref{eigenvalue of f_l} is a direct sequence of \eqref{error of gradient and Hessian}, \Cref{assumption: regularity}, and \Cref{assumption: Hessian}.
\end{proof}

{We generally require the difference length $l$ to be sufficiently small, which is specified by the following assumption.}

{\begin{assumption}\label{assumption: small l}
The difference length $l$ satisfies $dl\ll1, Ml^2d/2\ll 1$. 
\end{assumption}}

\subsection{Some technical results}
We first quote a classical result on the perturbation of the matrix eigenvalue problem, which will be used to assess the approximation of the unstable subspace of the Hessian.
\begin{lemma}[Davis-Kahan theorem \cite{yu2015useful}]\label{Davis-Kahan}
Let $\Sigma,\,\hat{\Sigma} \in \mathbb{R}^{d\times d}$ be symmetric matrices, 
with eigenvalues $\lambda_d \geqslant \dots \geqslant \lambda_1$ 
and $\hat{\lambda}_d \geqslant \dots \geqslant \hat{\lambda}_1$ respectively. 
Fix $1 \leqslant r \leqslant s \leqslant d$ and assume that 
$
\min(\lambda_r - \lambda_{r-1},\ \lambda_{s+1} - \lambda_s) > 0$,
where we define $\lambda_0 = -\infty$ and $\lambda_{d+1} = +\infty$. 
Let $p = s-r+1$, and let
$
V = [\vvec_r, \vvec_{r+1}, \dots, \vvec_s] \in \mathbb{R}^{d\times p}$ and 
$
\hat{V} = [\hat{\vvec}_r, \hat{\vvec}_{r+1}, \dots, \hat{\vvec}_s] \in \mathbb{R}^{d\times p}$ have orthonormal columns satisfying 
$\Sigma \vvec_j = \lambda_j \vvec_j$, $\hat{\Sigma} \hat{\vvec}_j = \hat{\lambda}_j\hat{\vvec}_j$, and $\|\vvec_j\|_2=\|\hat{\vvec}_j\|_2=1$,  for $j = r, r+1, \dots, s$. 
Then \[
\big\| VV^\top - \hat{V}\hat{V}^\top \big\|_F 
\ \leqslant\ 
\frac{
  2\sqrt{2}\ \min\!\big( p^{1/2} \|\hat{\Sigma} - \Sigma\|_2,\ \|\hat{\Sigma} - \Sigma\|_F \big)
}{
  \min(\lambda_r - \lambda_{r-1},\ \lambda_{s+1} - \lambda_s)
}.
\]
\end{lemma}

To help us prove the convergence of the saddle-search algorithm, we quote \cite[Lemma 4.3]{shi2025stochastic} on the convergence of a noisy recursive process.
\begin{lemma}\label{lemma: An and Bn}
Let $A(\x): \mathbb{R}^d \to \mathbb{R}$ be a function bounded from below, $B(\x): \mathbb{R}^d \to \mathbb{R}$ be a non-negative function, i.e., $B(\x) \geqslant 0$ for all $\x \in \mathbb{R}^d$, $\mathcal{F}(n)$ denote the filtration generated by the random variables $\mathcal{F}(n) = \{\x_0, \x_1, \dots, \x_n\}.$
With step size assumption in \Cref{assumption: step size sec 1}, suppose there exist constants $C_1, C_2 > 0$ such that 
\begin{equation}
\label{key equation for convergence}
    \mathbb{E}(A(\x(n+1))-A(\x(n))|\mathcal{F}(n))\leqslant-C_1\alpha(n)B(\x(n))+C_2\alpha(n)^2,
\end{equation}
then $A(\x(n))$ almost surely converges as $n\rightarrow \infty$. If $B(\x(n))$ also converges almost surely, then it almost surely converges to zero.
\end{lemma}

\section{Convergence of the derivative-free saddle search algorithm}\label{sec: Convergence of Algorithm}

We first study in \Cref{sec: convergence inner} the convergence of the eigenvector-search algorithm with the Hessian estimator. We also study the effect of dimensionality on the convergence in 
\Cref{sec: Dimension}, and study the convergence of the inner eigenvector-search algorithm \Cref{algorithm vector} in \Cref{sec: hessian-vector}. Then, based on the unstable directions output by the inner iteration, we analyze in \Cref{sec: convergence outer} the convergence of the outer iteration, that is, the derivative-free saddle-search algorithm (\Cref{algorithm}).

\subsection{Convergence of the inner derivative-free eigenvector-search algorithm with zeroth-order Hessian estimator}\label{sec: convergence inner}

For the saddle point search, the important aspect is to identify the unstable directions. Thus, the convergence of the inner iteration focuses on the convergence of the projection operator on the unstable subspaces. We note that this analysis largely follows the approach used in \cite{shi2025stochastic}.

\begin{lemma}\label{Lemma: eigenvector search}
Under \Cref{assumption: step size sec 1} and \Cref{assumption: small l}, the limit points of $\mathbf{v}(n)$ generated by \eqref{stochastic eigenvector} with $l(n)\equiv l$ almost surely lie in the eigenspace of $\nabla^2f_l(\x)$. The limit points, $\vvec_\infty \notin span\{\vvec_{l,1},\cdots,\vvec_{l,\bar{k}-1}\},$ of the iteration points $\vvec(n)$ generated by \eqref{stochastic eigenvector larger}, almost surely lie in the eigenspace of $\nabla^2 f_l(\x)$, provided that $\vvec_{l,i},i=1,\cdots,\bar{k}-1$ are eigenvectors of $\nabla^2 f_l(\x)$ and $\vvec(0)$ is orthogonal to $\vvec_i,i=1,\cdots,\bar{k}-1$.
\end{lemma}
\begin{proof}
    This is a direct sequence of \Cref{Lemma: Almost unbiased estimator} and \cite[Proposition 4.9]{shi2025stochastic}.
\end{proof}

\begin{remark}
    Consider the Rayleigh quotient problem
$$
    \min_\vvec\  \vvec^\top \nabla^2f_l(\x) \vvec
    \text{  s.t. } \Vert \vvec  \Vert^2=1, \vvec_{l,j}^T\vvec=0, \;j=1,\cdots, \bar{k}-1.
$$
The solution to the above problem gives the eigenvector $\vvec_{l,\bar{k}}$ corresponding to the $\bar{k}$-th smallest eigenvalue of $\nabla^2 f_l(\x)$, where $\vvec_{l,i},i=1,\cdots,\bar{k}-1$ is the eigenvector corresponding to the $i$-th smallest eigenvalue. If $\lambda_{\bar{k}}$ is simple,  $\vvec_{l,\bar{k}}$ serves as the global minimizer, and $\vvec_{l,i}, i>\bar{k}$ serve as the unstable saddle point. Even in the worst-case scenario, where the algorithm converges to a saddle point or the eigenvector corresponding to positive eigenvalues, we can repeat the eigenvector search to identify the unstable directions. Thus, in the following, we assume that the eigenspace in \Cref{Lemma: eigenvector search} is actually spanned by the $k$ smallest eigenvectors of $\nabla^2f_l(\x)$.
\end{remark}

{Denote the eigenvectors corresponding to the $k$ smallest eigenvalues of $\nabla^2f(\x)$ and $\nabla^2f_l(\x)$ by $V=[\vvec_1(\x),\ldots,\vvec_k(\x)]$ and $V_l=[\vvec_{1,l}(\x),\ldots,\vvec_{k,l}(\x)]$, respectively, and denote the outputs of \eqref{stochastic eigenvector} and \eqref{stochastic eigenvector larger} by $\bar{V}=[\bar{\vvec}_1,\ldots,\bar{\vvec}_k]$.}
\begin{lemma}
\label{lem:quote}
Suppose \Cref{assumption: step size sec 1} holds. Given an arbitrary tolerance $\epsilon_{\vvec}>0$, $\bar{V}$ achieves the tolerance \begin{equation}\label{eq: error between output and vl}
\left\Vert \bar{V}\bar{V}^\top - V_{l} V_{l}^\top
\right \Vert_2 \leqslant \epsilon_\vvec,
\end{equation}
in a finite number of iterations, almost surely.
\end{lemma}
\begin{proof}
    This is a direct sequence of \cite[Theorem 4.13]{shi2025stochastic}.
\end{proof}

\begin{corollary}
    Suppose \Cref{assumption: step size sec 1}, \Cref{assumption: regularity}, and \Cref{assumption: Hessian} hold. If $\bar{V}$ satisfies the tolerance in \eqref{eq: error between output and vl}, then it is also a good approximation of $V$
 if the difference length $l$ is small. Specifically,
$$
    \left\Vert {\bar{V}\bar{V}^\top - V V^\top}
    \right \Vert_2 \leqslant \frac{\sqrt{2k}Ml^2d}{4\mu}
+\epsilon_\vvec.
$$
\end{corollary}
\begin{proof}
By setting $\Sigma=\nabla^2f(\x)$, $\hat{\Sigma}=\nabla^2f_l(\mathbf{x})$, $s=k$, $r=1$ in Lemma \ref{Davis-Kahan}, we have $\min(\lambda_{k+1}-\lambda_k,\lambda_1-\lambda_0)\geqslant 2\mu$, and
$$
\begin{aligned}
   \|V_lV_l^\top-VV^\top\|_2&\leqslant \left\Vert V_lV_l^\top-VV^\top \right \Vert_F \leqslant \frac{\sqrt{2k}}{2\mu}\Vert\nabla^2 f_l(\x)-\nabla ^2 f(\mathbf{\x})\Vert_2 \leqslant \frac{\sqrt{2k}Ml^2d}{4\mu},
\end{aligned}
$$
$$\|\bar{V}\bar{V}^\top-VV^\top\|_2\leqslant \| \bar{V}\bar{V}^\top-V_lV_l^\top\|_2+\|V_lV_l^\top-VV^\top\|_2\leqslant \frac{\sqrt{2k}Ml^2d}{4\mu}
+\epsilon_\vvec.
$$
\end{proof}

\subsection{The dimensional dependence of the error}\label{sec: Dimension}
We should note that the zeroth-order Hessian estimator in \eqref{stochastic eigenvector} and \eqref{stochastic eigenvector larger} suffers from the curse of dimensionality. By ignoring the $O(l^4d^6)$ term which can be controlled by a small enough $l$ in \Cref{Lemma: Almost unbiased estimator}, we have 
\begin{equation}
\begin{aligned}
\mathbb{E}_\rvec[\Vert\Hvec(\x,\rvec,l)\Vert_2^2]
\lesssim \frac{d(d+2)(d^2+12d+33)M^2}{2}=O(d^4).
\end{aligned}
\label{curse of dimensionality}
\end{equation}

    In fact, $\Hvec(\x,\rvec,l)$ contains only one-dimensional curvature information to approximate the $d\times d$ Hessian. It is therefore not surprising that the variation increases quickly as the dimension $d$ grows. Then, we show that this large variation will slow down the convergence. Without loss of generality, we assume that $\mathbb{E}_\rvec (\Vert \Hvec(\x,\rvec,l) \Vert_2^2)\leqslant C_3d^4$. Consider the iteration in \eqref{stochastic eigenvector}, and define Rayleigh Quotient and the squared norm of its Riemannian gradient, as $A(\vvec)=\vvec^T\nabla^2f(\x)\vvec$ and $B(\vvec)=\Vert (\I-\vvec \vvec^\top) \nabla^2f(\x) \vvec \Vert^2_2$, then 
    $$
    \begin{aligned}
        &A(\vvec(n+1))-A(\vvec(n))
        \\
        &\leqslant -2\alpha(n)\langle (\I-\vvec(n) \vvec(n)^\top) \Hvec(\x,\rvec(n),l(n))\vvec(n),(\I-\vvec(n) \vvec(n)^\top) \nabla^2f_l(\x) \vvec(n) \rangle\\
        &\quad+M_1 \alpha(n)^2\Vert\Hvec(\x,\rvec(n),l(n))\Vert^2_2,
        \end{aligned}
    $$
where $M_1$ is a positive constant dependent on $M$. Taking expectations on both sides, 
$$
\begin{aligned}
    &\mathbb{E}[A(\vvec(n+1))-A(\vvec(n))|\mathcal{F}(n)]\\
    &\leqslant -2\alpha(n)B(\vvec(n))+M_1 \alpha(n)^2\mathbb{E}[\Vert\Hvec(\x,\rvec(n),l(n))\Vert_2^2]\\
    &\leqslant -2\alpha(n)B(\vvec(n))+M_1C_3 \alpha(n)^2d^4.
\end{aligned}
$$
By taking expectation with respect to $\mathcal{F}(n)$, we get that
$$
    \mathbb{E}[A(\vvec(n+1)]-\mathbb{E}[A(\vvec(n))]\leqslant -2\alpha(n)\mathbb{E}[B(\vvec(n))]+M_1C_3 \alpha(n)^2d^4.
$$
Using the telescoping argument, we obtain 
\begin{equation}
\begin{aligned}
\min_{1\leqslant i\leqslant n} \mathbb{E}[B(\vvec(i))]&\leqslant \frac{A(\vvec(0))-\mathbb{E}[A(\vvec(n+1))]+M_1C_3\sum_{i=0}^\infty\alpha(i)^2 d^4}{2\sum_{i=0}^n \alpha(i)}\\
&\leqslant \frac{A(\vvec(0))+M+M_1C_3\sum_{i=0}^\infty\alpha(i)^2 d^4}{2\sum_{i=0}^n \alpha(i)}.
\end{aligned}
\label{convergence rate versus variation}
\end{equation}
Hence, a larger dimension $d$ slows down the convergence. Moreover, in \Cref{sec: Rosenbrock}, we show that the large variation also leads to potential numerical instability. Fortunately, we do not need the full Hessian information. In \eqref{stochastic eigenvector} and \eqref{stochastic eigenvector larger}, it suffices to approximate the Hessian-vector product, $\nabla^2 f(\x)\vvec$, which is also a key idea in \cite{zhangdu2012} for avoiding any explicit Hessian evaluation. To this end, $ H_\vvec(\x,\rvec,l)$
as defined by \eqref{eq:Hv}
could be used to approximate $\nabla^2f(\x)\vvec$.
\begin{lemma}\label{Hv expectation and variation}
Suppose \Cref{assumption: regularity} and \Cref{assumption: small l} hold, 
then 
$$\Vert\mathbb{E}_\rvec [H_\vvec(\x,\rvec,l)]-\nabla^2 f(\x)\vvec\Vert_2\leqslant Ml^2d,$$ 
$$
    \begin{aligned}
    \mathbb{E}_\rvec [\Vert H_\vvec(\x,\rvec,l) \Vert_2^2]&\leqslant 4(d+2)M^2+\frac{M^2d(d+2)l^4}{9}\\
    &\quad+\frac{M^2d(d+2)(d+4)(d+6)l^4}{9}+O(l^8d^8)\leqslant 10M^2d.
    \end{aligned}
    $$
\end{lemma}

\begin{proof} 
The proof is given in the Appendix.
\end{proof}

\begin{figure}
    \centering \includegraphics[width=0.5\linewidth]{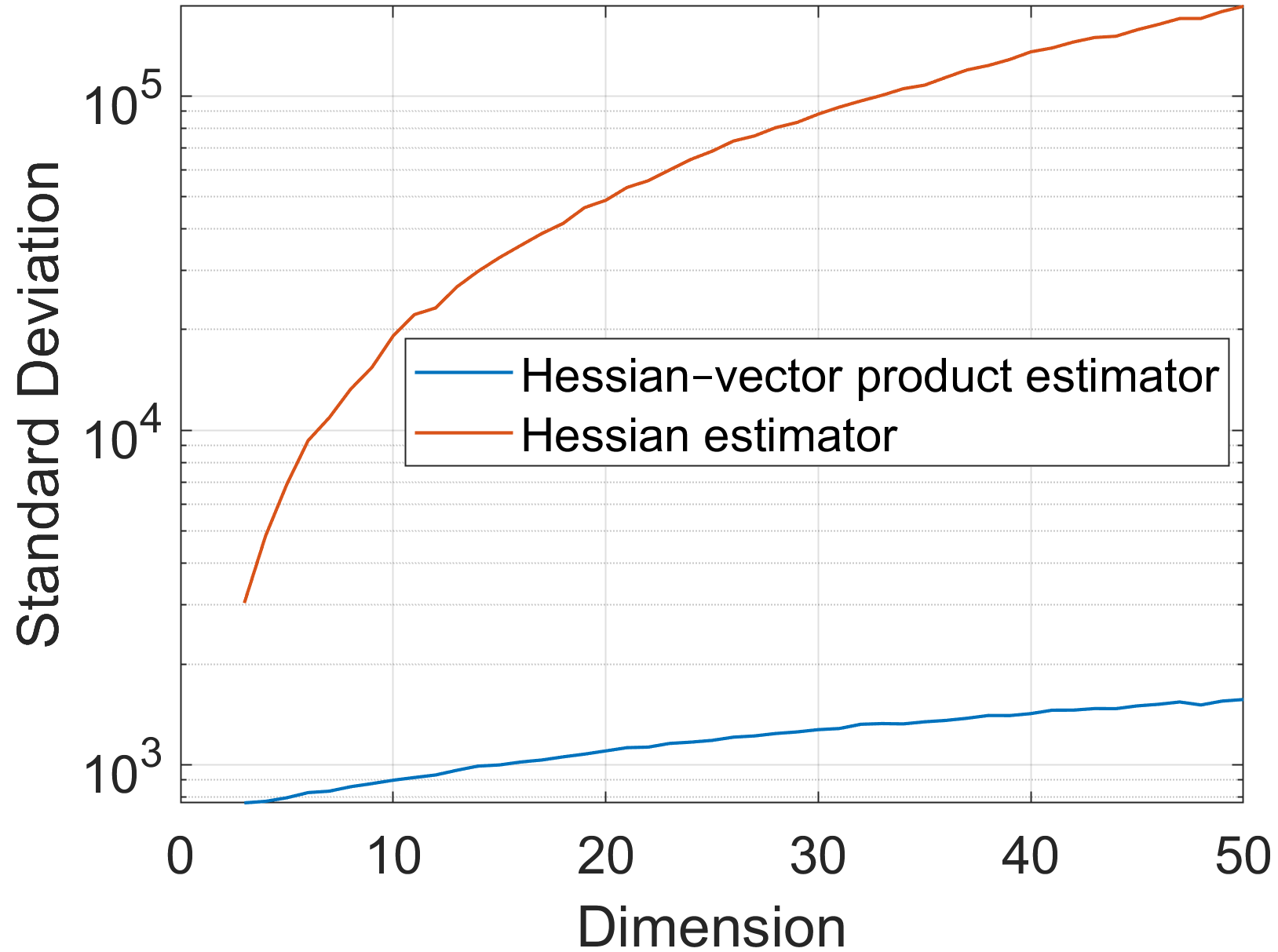}
    \caption{The standard deviations are computed over 10000 samples of $\Hvec \vvec$ (Hessian estimator) and $H_\vvec$ (Hessian-vector estimator), where $\vvec$ is a fixed vector drawn from $\mathcal{N}(\mathbf{0},\I)$.}
    \label{fig: placeholder}
\end{figure}
The use of $H_\vvec$ significantly reduces the variation in \eqref{curse of dimensionality} from $O(d^4)$ to $O(d)$ (also see \Cref{fig: placeholder}), and accelerates the convergence of the eigenvector search. 

\subsection{Convergence of the inner derivative-free eigenvector-search algorithm with zeroth-order Hessian-vector estimator}\label{sec: hessian-vector}
The following proposition demonstrates that the output of the \Cref{algorithm vector} with Hessian-vector estimator $H_\vvec$ can serve as approximations to the unstable directions of both the original function $f(\x)$ and the Gaussian smoothed function $f_{l_0}(\x)$ with a small $l_0$. 

Denote the output of the \Cref{algorithm vector} by $\tilde V=[\tilde{\vvec}_1,\cdots,\tilde{\vvec}_k]$.

\begin{proposition}\label{proposition eigenvecor search}
Suppose the \Cref{assumption: step size sec 1}, \Cref{assumption: decay difference length}, and \Cref{assumption: regularity} hold. Then, for any $\epsilon_\vvec$, there almost surely exists a finite $n^{max}_\vvec$ (that depends on $\epsilon_\vvec$ and on the random realization of the samples $\rvec(n)$) such that the output $\tilde V$ of \Cref{algorithm vector} satisfies $\left\Vert \tilde V \tilde V^\top-V V^\top\right \Vert_2 \leqslant \epsilon_\vvec
$
and
$\left\Vert \tilde V \tilde V^\top-V_l V_l^\top \right \Vert_2 \leqslant \epsilon_\vvec+\frac{\sqrt{2k}Ml_0^2d}{4\mu}.
$
\end{proposition}

\begin{proof}
    The proof follows the same line as \cite[Theorem 4.13]{shi2025stochastic}, which itself relies on \cite[Proposition 4.8]{shi2025stochastic}, that is, the single-eigenvector iteration in \eqref{stochastic eigenvector} converges almost surely to the true unstable eigenvector. By choosing a sufficiently large $n_{max}$, we can ensure that the output of the algorithm is close to the true eigenvector. A key distinction here is that \cite[Proposition 4.8]{shi2025stochastic} is proved under the assumption of an unbiased estimator, whereas the estimate $H_\vvec$ is biased. We control the bias using the condition $l(n)\leqslant L\sqrt{\alpha(n)}$
    so that the additional error introduced by the bias becomes asymptotically negligible, and the almost-sure convergence statement also carries over to our setting. To show this specifically, for any $\vvec$ with $\|\vvec\|_2=1$, we define the Rayleigh Quotient as $A(\vvec)=\vvec^T\nabla^2f(\x)\vvec$, then 
    $$
    \begin{aligned}
        A&(\vvec(n+1))-A(\vvec(n))
        \\
        =&\vvec(n+1)^\top\nabla^2f(\x)\vvec(n+1)-\vvec(n)^\top\nabla^2f(\x)\vvec(n)\\
        &-\langle(\I-\vvec(n) \vvec(n)^\top) H_{\vvec(n)}(\x,\rvec(n),l(n)),\vvec(n)\vvec(n)^\top\nabla^2f(\x) \vvec(n)\rangle\\
        \leqslant& -2\alpha(n)\langle (\I-\vvec(n) \vvec(n)^\top) H_{\vvec(n)}(\x,\rvec(n),l(n)),(\I-\vvec(n) \vvec(n)^\top) \nabla^2f(\x) \vvec(n) \rangle\\
        &+M_2 \alpha(n)^2\Vert H_{\vvec(n)}(\x,\rvec(n),l(n))\Vert^2_2,
        \end{aligned}
    $$
where $M_2$ is a positive constant only dependent on $M$. Taking expectations on both sides, and utilizing \Cref{Hv expectation and variation}, we have
$$
\begin{aligned}
    &\mathbb{E}(A(\vvec(n+1))-A(\vvec(n))|\mathcal{F}(n))\\
    &\leqslant -2\alpha(n)B(\vvec(n))+M_2 \alpha(n)^2\mathbb{E}_\rvec[\Vert H_{\vvec(n)} \Vert_2^2]\\
    &\quad+2\alpha(n)\langle (\I-\vvec(n) \vvec(n)^\top) (\nabla^2f(\x)\vvec(n)-\mathbb{E}_\rvec[H_{\vvec(n)}]),(\I-\vvec(n) \vvec(n)^\top) \nabla^2f(\x) \vvec(n) \rangle\\
    &\leqslant -2\alpha(n)B(\vvec(n))+ \left(10dM_2M^2+2L^2M^2d\right)\alpha(n)^2,
\end{aligned}
$$
where $B(\vvec)=\Vert (\I-\vvec \vvec^\top) \nabla^2f(\x) \vvec \Vert^2_2$. Then, same as the proof in \cite[Proposition 4.8]{shi2025stochastic}, we can show that $A(\vvec(n))$ and $B(\vvec(n))$ satisfy the condition in \Cref{lemma: An and Bn}, and hence the almost sure convergence of $\vvec(n)$ to the true unstable eigenvector holds.

\end{proof}

{From \Cref{proposition eigenvecor search}, we know that the iterates
$\vvec(n)$ of the inner eigenvector search converges almost surely to the true
unstable eigenvector. However, the number of iterations required is a random quantity that depends on the sample path. In the remainder of this subsection, we propose a more automatic stopping criterion and, for brevity, restrict to the case $k=1$. By \cite[Theorem 4.13]{shi2025stochastic}, the distance between $\vvec(n)$ and the true unstable eigenvector can be controlled by the residual
$\|(\I - \vvec(n)\vvec(n)^\top)\nabla^2 f(\x)\vvec(n)\|_2$. This
motivates us to use the zeroth-order residual $\|(\I - \vvec(n)\vvec(n)^\top)\bar H_{\vvec(n)}\| < \epsilon$ as a potential stopping criterion, where 
$$\bar H_{\vvec(n)}:= \frac{1}{m}\sum_{i=1}^m H_{\vvec(n)}\bigl(\x,\rvec(i),l(n)\bigr).$$

In what follows, we first prove that the stopping criterion almost surely holds in a finite number of iterations. We also prove that this stopping criterion with an increasing $m$ guarantees with a positive probability that, at each inner eigenvector search, $\|(\I - \vvec(n)\vvec(n)^\top)\nabla^2 f(\x)\vvec(n)\|_2$ is sufficiently small so that the returned vector is sufficiently close to the true unstable eigenvector by \cite[Theorem 4.13]{shi2025stochastic}.}
\begin{proposition}
    Under the same conditions as in \cref{proposition eigenvecor search}, $\forall m\geqslant1,\epsilon>0$, the iteration $\vvec(n)$ satisfies the stopping criterion $\|(\I-\vvec(n)\vvec(n)^\top)\bar H_{\vvec(n)}\|_2<\epsilon$ in a finite number of iterations almost surely.
\end{proposition}

\begin{proof}
    Define the termination
    time $T:=\inf\{n\geqslant 1,\|(\I-\vvec(n)\vvec(n)^\top)\bar H_{\vvec(n)}\|_2<\epsilon \}$. 
For each $n$, we decompose
\begin{equation}\label{eq: decomposition}
\begin{aligned}    
    (\I-\vvec(n)\vvec(n)^\top)\bar H_{\vvec(n)}
    &=
    (\I-\vvec(n)\vvec(n)^\top)\nabla^2f(\x)\vvec(n)\\
    & \quad +(\I-\vvec(n)\vvec(n)^\top)(\mathbb{E}_\rvec[\bar H_{\vvec(n)}]-\nabla^2f(\x)\vvec(n))\\
    &\quad 
    +(\I-\vvec(n)\vvec(n)^\top)(\bar H_{\vvec(n)}-\mathbb{E}_\rvec[\bar H_{\vvec(n)}])\,.
    \end{aligned}
\end{equation}
    By \Cref{proposition eigenvecor search}, for almost every random event $\omega$, there exists $N(\omega)$ such that, for all $n > N_1(\omega)$, $\|(\I-\vvec(n)\vvec(n)^\top)\nabla^2f(\x)\vvec(n)\|_2<\epsilon/4$. Note that $l(n) \leqslant L \sqrt{\alpha(n)}\to 0$, by \cref{Hv expectation and variation}, there exists a deterministic $N_2$ such that, for all $n > N_2$ (so that $l(n)$ is small enough), $\|(\I-\vvec(n)\vvec(n)^\top)(\mathbb{E}_\rvec[\bar H_{\vvec(n)}]-\nabla^2f(\x)\vvec(n))\|_2<\epsilon/4$. Next, by the Taylor expansion of $f$ at $\x$, we have for each fixed $\vvec$ and $l$,
$$
\begin{aligned}
H_\vvec&=\frac{F(\x+l\vvec,\rvec,l)-F(\x-l\vvec,\rvec,l)}{2l}=\rvec\rvec^\top\nabla^2f(\x)\vvec+O(l^2),
\end{aligned}
$$  
where the $O(l^2)$ term is uniform in $\rvec$ under the regularity assumptions on $f$. Since each $\rvec(i)$ is drawn from the standard normal distribution, when $l(n)$ is small, there exist $N_3\in\mathbb N$ and a constant $\delta_p>0$ such that, for all $n > N_3$, we have $\mathbb{P}(\|(\I-\vvec(n)\vvec(n)^\top)(\bar H_{\vvec(n)}-\mathbb{E}_\rvec[\bar H_{\vvec(n)}])\|_2<\epsilon/2)\geqslant\delta_p$. Hence, $\forall n>N_4=\max(N_1,N_2,N_3)$,  we have $\mathbb{P}(T>n)\leqslant (1-\delta_p)^{n-N_4}$, which means $\mathbb{P}(T=\infty)=0$.
\end{proof}

\begin{proposition}\label{pro: trust region}
Let $T$ be the termination time, i.e., the stopping criterion $\|(\I-\vvec(T)\vvec(T)^\top)\bar H_{\vvec(T)}\|<\epsilon$ holds, and $m$ is large enough such that $10M^2d<m\epsilon^2$, then $$\mathbb{P}\left(\|(\I-\vvec(T)\vvec(T)^\top)\nabla^2f(\x)\vvec(T)\|_2<2\epsilon+Ml(T)^2d\right)\geqslant 1-\frac{10M^2d}{m\epsilon^2}.$$ 
\end{proposition}

\begin{proof}
From \eqref{eq: decomposition}, we have that 
$$
\begin{aligned}
    &\|(\I-\vvec(T)\vvec(T)^\top)\nabla^2f(\x)\vvec(T)\|_2\\
    &\leqslant  \|(\I-\vvec(T)\vvec(T)^\top)\bar H_{\vvec(T)}\|_2+\|(\I-\vvec(T)\vvec(T)^\top)(\mathbb{E}_\rvec[\bar H_{\vvec(T)}]-\nabla^2f(\x)\vvec(T))\|_2
    \\
    &\quad+\|(\I-\vvec(T)\vvec(T)^\top)(\bar H_{\vvec(T)}-\mathbb{E}_\rvec[\bar H_{\vvec(T)}])\|_2\\
    &<  \epsilon + Ml(T)^2 d+\|(\I-\vvec(T)\vvec(T)^\top)(\bar H_{\vvec(T)}-\mathbb{E}_\rvec[\bar H_{\vvec(T)}])\|_2.
\end{aligned}
$$
Note that 
{
$$
\begin{aligned}
&\mathbb{E}_{\rvec}[\|(\I-\vvec(T)\vvec(T)^\top)(\bar H_{\vvec(T)}-\mathbb{E}_\rvec[\bar H_{\vvec(T)}])\|_2^2]\leqslant \mathbb{E}_{\rvec}[\|\bar H_{\vvec(T)}-\mathbb{E}_\rvec[\bar H_{\vvec(T)}]\|_2^2]\\
&\leqslant \frac{\mathbb{E}_{\rvec}[\|H_{\vvec(T)}-\mathbb{E}_\rvec[ H_{\vvec(T)}]\|_2^2]}{m}\leqslant \frac{\mathbb{E}_{\rvec}[\|H_{\vvec(T)}\|_2^2]}{m}\leqslant\frac{10M^2d}{m},
\end{aligned}
$$
where the first inequality follows from the bound $\|\I-\vvec(T)\vvec(T)^\top\|_2\leqslant 1$, the second inequality holds because $\bar H_\vvec$ is the average of $m$ independent copies of $H_\vvec$, and the last inequality follows from \cref{Hv expectation and variation}.} From the Chebyshev inequality, we get
$$\mathbb{P}(\|(\I-\vvec(T)\vvec(T)^\top)(\bar H_{\vvec(T)}-\mathbb{E}_\rvec[\bar H_{\vvec(T)}])\|_2\geqslant \epsilon)\leqslant \frac{10M^2d}{m\epsilon^2}.$$
\end{proof}

\begin{corollary}\label{corollary: stopping criterion}
Suppose the same conditions in \cref{proposition eigenvecor search} hold. At the $n$-th iteration of the outer saddle algorithm, choose the sample size
$m_n = O(n^p)$ with some $p>1$ in the stopping criterion of the inner
eigenvector search. That is, in the inner loop, we form $\bar H_{\vvec}:= \frac{1}{m_n}\sum_{i=1}^{m_n} H_{\vvec}\bigl(\x_n,\rvec(i),l\bigr), \rvec(i)\sim \mathcal{N}(\mathbf{0},\I)$
and stop as soon as $\|(\I - \vvec\vvec^\top)\,\bar H_{\vvec}\| < \epsilon.$ Then, there exists a positive probability that the vector $\vvec_n(T)$ returned by the inner algorithm satisfies
\[
    \|(\I - \vvec_n(T)\vvec_n(T)^\top)\nabla^2 f(\x_n)\vvec_n(T)\|_2
    < 2\epsilon + M l(0)^2 d, \forall n\in\mathbb N.
\]
\end{corollary}

\begin{proof}
Note that $l(0)\geqslant l(T)$ by \Cref{assumption: decay difference length}. By the \cref{pro: trust region}, $$
\mathbb{P}(
  \|(\I-\vvec_n(T)\vvec_n(T)^\top)\nabla^2 f(\x_n)\vvec_n(T)\|_2
  < 2\epsilon + M dl(0)^2)\geqslant
1-\frac{10M^2d}{m_n\epsilon^2}.
$$
Moreover, by choosing $m_n = O(n^p)$ with $p>1$, the probability that $$\|(\I-\vvec_n(T)\vvec_n(T)^\top)\nabla^2 f(\x_n)\vvec_n(T)\|_2< 2\epsilon + M dl(0)^2 \text{ for every } n\in\mathbb N$$ holds is bounded from below by $\prod_{n=1}^\infty \left(1 - O\left(\frac{1}{n^p}\right)\right)>0$. 
\end{proof}

\subsection{Convergence of the outer derivative-free saddle search algorithm}\label{sec: convergence outer}
From \Cref{proposition eigenvecor search}, the output $\{\tilde{\vvec}_i\}_1^k$, of  \Cref{algorithm vector} can almost surely approximate the unstable directions of both $f(\x)$ (i.e., $\{\vvec_i\}_1^k$) and $f_l(\x)$ (i.e., $\{\vvec_{l,i}\}_1^k$) with arbitrarily small error in a finite number of iterations if $l$ is small enough. It is worth noting that $\x(n)$ is generally close to $\x(n+1)$ so that $\{\vvec_i(n)\}_1^k$ serves as a good initial condition for $\{\vvec_i(n+1)\}_1^k$, thus, a relatively small 
$n_\vvec^{max}$ (e.g., 
$n_\vvec^{max}=10$) is often sufficient to obtain $\vvec(n+1)$ with satisfactory accuracy in practice. Thus, rather than imposing the stopping criterion in \cref{corollary: stopping criterion}, which
requires many function evaluations, we recommend in practice simply imposing the maximum number of iterations $n_\vvec^{max}$ to reduce the computational cost. Nevertheless, the stopping criterion in \cref{corollary: stopping criterion} guarantees that, with positive probability, each inner eigenvector search returns a good output. In this subsection, we work under the condition that this event occurs, i.e., we assume that the vectors returned by the inner loop are sufficiently close to the unstable eigenvectors, as made precise in the following assumption.
\begin{assumption}\label{assumption on theta}
     There is a small parameter $\theta$ such that, at any $n$-th outer iteration, 
\begin{equation}\label{eq:theta conditions}
    \theta>\frac{\sqrt{2k}Ml(n)^2d}{4\mu}>0,   \text{ and }
 (1 - \sqrt{\theta})\mu_{l(n)} - M\theta - 5M\sqrt{\theta} > 0 
\end{equation}  
holds, where $\mu_{l(n)}=\mu-Ml(n)^2d/2$, and the following conditions hold almost surely,
\begin{equation}\label{eq: precision of the eigenvector search}
\min(\Vert\tilde V(n) \tilde V(n)^\top-V(n) V(n)^\top\Vert_2^2,\Vert \tilde V(n) \tilde V(n)^\top- V_{l(n)} V_{l(n)}^\top \Vert_2^2)\leqslant \theta,
\end{equation}
where $\tilde V(n) \tilde V(n)^\top=\sum_{i=1}^k \tilde{\vvec}_i(\x(n))\tilde{\vvec}_i(\x(n))^\top$
is the projection operator with respect to,  respectively, the output of the \Cref{algorithm vector} with $\x=\x(n)$. Likewise, the operators $V(n)V(n)^\top=\sum_{i=1}^k \vvec_i(\x(n)) \vvec_i(\x(n))^\top$, $V_{l(n)}V_{l(n)}^\top=\sum_{i=1}^k \vvec_{{l(n)},i}(\x(n)) \vvec_{{l(n)},i}(\x(n))^\top$ are projections with respect to the $k$ eigenvectors corresponding to the $k$ smallest eigenvalues of $\nabla^2 f(\x(n))$ and $\nabla^2 f_{l(n)}(\x(n))$.
\end{assumption}

\begin{theorem}[Decay $l(n)$ with decay step size $\alpha(n)$]\label{decay ln and an}
    Suppose the \Cref{assumption: step size sec 1}, \Cref{assumption: decay difference length}, \Cref{assumption: regularity}, \Cref{assumption: Hessian}, and \Cref{assumption on theta} hold. If the initial point \( \x(0) \) is sufficiently close to \( \x^* \), and that \( \sum_{n=0}^\infty \alpha(n)^2 \) is sufficiently small (e.g., 
    \( \alpha(n) = \frac{\gamma}{(n+m)^p} \) with some
    \( p \in (1/2, 1] \),
    and a large enough \( m \)), the following boundedness event occurs with a high probability,   \begin{equation}\label{eq:boundness event mu}
        E^\infty = \{ \x(n) \in U, \forall n \},
       \;
        U=\left\{\x,\Vert \x-\x^* \Vert_2 \leqslant \min \left( \frac{(1-\sqrt{\theta})\mu}{M}-\theta-5\sqrt{\theta},\delta\right)\right\}. \end{equation}
    Conditioned on $E^\infty$, given any tolerance $\epsilon_\x>0$, the \Cref{algorithm} almost surely reaches the $\epsilon_\x$-neighborhood of $\x^*$ in a finite number of iterations.
\end{theorem}
\begin{proof}
    The proof is given in the Appendix. Since $F(\x,\rvec,l)$ is a biased gradient estimator with a bias term of order $O(l^2)$ (\Cref{Lemma: Almost unbiased estimator}), we can use the bound on the decaying $l(n)$ with the step size $\alpha(n)$ (\Cref{assumption: decay difference length}) 
    to control this biased error. 
\end{proof}

However, if \(l(n)\), appearing as a denominator in \(F(\x(n), \rvec(n), l(n))\), is too small,  then \(F(\x(n), \rvec(n), l(n))\) may be dominated by round-off error and fail to approximate the gradient. Thus, in the rest of the section, let us consider the case with a small enough but constant $l(n)\equiv l>0$.

\begin{lemma}\label{lemma: convergence to xl}
Suppose the \Cref{assumption: step size sec 1}, \Cref{assumption: regularity}, \Cref{assumption: Hessian}, \Cref{assumption: small l}, and \Cref{assumption on theta} hold. If the initial point \( \x(0) \) is sufficiently close to \( \x^* \) and \( \sum_{n=0}^\infty \alpha(n)^2 \) is sufficiently small, then, the boundedness event 
$$E_l^\infty = \{ \x(n) \in U, \forall n \},
       \;
        U_l=\left\{\x,\Vert \x-\x^* \Vert_2 \leqslant \min \left( \frac{(1-\sqrt{\theta})\mu_l}{M}-\theta-5\sqrt{\theta},\delta\right)\right\} $$
        occurs with a high probability. Conditioned on $E_l^\infty$, given any tolerance $\epsilon_\x>0$, the \Cref{algorithm} almost surely reaches the $\epsilon_\x$-neighborhood of $\x^*_l$ in a finite number of iterations.
\end{lemma}
\begin{proof}
    Note that $F(\x(n),\rvec(n),l)$ is the unbiased gradient estimator of $f_l(\x(n))$. Therefore, the proof then follows the same steps as in that of \Cref{decay ln and an} by assuming that the biased error is zero.
\end{proof}

\Cref{lemma: convergence to xl} shows that the iteration points have a large probability to converge to \(\x_l^*\). It is therefore natural to estimate the distance between \(\x_l^*\) and the target saddle point \(\x^*\). From \eqref{error of gradient and Hessian}, the difference between the gradients and Hessians of \(\nabla f(\x)\) and \(\nabla f_l(\x)\) are of \(O(l^2)\). Consequently, we can expect that their zeros, namely \(\x^*\) and \(\x_l^*\), are also separated by a distance of \(O(l^2)\), as made precise in \Cref{lemma: distance between two saddle points}.

\begin{lemma}\label{lemma: distance between two saddle points}
Suppose \Cref{assumption: regularity} and \Cref{assumption: Hessian} hold. If $l>0$ is a sufficiently small parameter such that $
\mu_l>0,\frac{Ml^2d}{\mu_l}=\delta_l\leqslant \frac{\delta}{2},\frac{M\delta_l}{\mu_l}\leqslant\frac{1}{2}.
$
Then, there exists an $\x_l^*\in \{\x,\Vert \x-\x^* \Vert_2\leqslant \delta_l\}$ which is an index-$k$ saddle point of $f_l$, such that  \begin{equation}\label{distance between two saddles}
        \Vert \x_l^*-\x^* \Vert_2^2\leqslant \delta_l^2=\left(\frac{Ml^2d}{\mu_l}\right)^2=O(l^4).
    \end{equation}
\end{lemma}
\begin{proof}
    We borrow the technique described in \cite[Section 5.4]{kelley1995iterative} to estimate the distance between the zeros of two functions whose gradients and Hessians are close. Define the Newton iteration map as 
    $
    T(\x)=\x-(\nabla^2f_l(\x^*))^{-1}\nabla f_l(\x)$. For $\x,\mathbf{y} \in \{\x,\Vert \x-\x^* \Vert_2\leqslant \delta_l\}$, we have
    $$
    \begin{aligned}
    \|T(\x)-T(\mathbf{y})\|_2&=\|(\x-\mathbf{y})-(\nabla^2f_l(\x^*))^{-1}(\nabla f_l(\x)-\nabla f_l (\mathbf{y}))\|_2\\
    &=\left\|\left[(\nabla^2 f_l(\x^*))^{-1}\int_0^1 (\nabla^2 f_l(\x^*)-\nabla^2 f_l(\mathbf{y}+t(\x-\mathbf{y})))\mathrm{d}t\right](\x-\mathbf{y})\right\|_2\\
    &\leqslant \frac{M\delta_l}{\mu_l}\Vert \x-\y \Vert_2 \leqslant\frac{\Vert \x-\y \Vert_2}{2}. 
    \end{aligned}
    $$
Form \eqref{error of gradient and Hessian}, we have $\Vert \nabla f_l(\x^*) \Vert_2=\Vert \nabla f_l(\x^*)-\nabla f(\x^*) \Vert_2\leqslant \frac{Ml^2d}{2}$; from \eqref{eigenvalue of f_l}, we have that $\|(\nabla^2f_l(\x^*))^{-1}\|_2\leqslant \frac{1}{\mu_l}$. Thus, if $\x\in \{\x,\Vert \x-\x^* \Vert_2\leqslant \delta_l\}$,
$$
\begin{aligned}
&\Vert T(\x)-\x^*\Vert_2=\|\x-(\nabla^2f_l(\x^*))^{-1}\nabla f_l(\x)-\x^*\|_2\\
&\leqslant \underbrace{\|\x-(\nabla^2f_l(\x^*))^{-1}(\nabla f_l(\x)-\nabla f_l(\x^*))-\x^*\|_2}_{=\|T(\x)-T(\x^*)\|_2}+\underbrace{\| (\nabla^2f_l(\x^*))^{-1}\nabla f_l(\x^*) \|_2}_{\leqslant \| (\nabla^2f_l(\x^*))^{-1}\|_2\|\nabla f_l(\x^*)\|_2}\\
&\leqslant \frac{1}{2}\Vert \x-\x^* \Vert_2+\frac{Ml^2d}{2\mu_l}\leqslant \frac{\delta_l}{2}+\frac{\delta_l}{2}=\delta_l, 
\end{aligned}
$$
which means $T(\x)$ is a contraction mapping in $\{\x,\|\x-\x^*\|_2\leqslant \delta_l\}$. According to the Banach fixed point theorem, there exists an $\x_l^* \in \{\x,\Vert \x-\x^* \Vert_2\leqslant \delta_l\}$ such that $T(\x_l^*)=\x_l^*$, which means $\x_l^*$ is a critical point of $f_l(\x)$. Since $\x_l^* \in \{\x,\Vert \x-\x^* \Vert_2\leqslant \delta_l\}\subset \{\x,\Vert \x-\x^* \Vert_2\leqslant \delta\}$, we have \eqref{distance between two saddles}. From \eqref{eigenvalue of f_l}, we know that $\x_l^*$ is an index-$k$ saddle point of $f_l(\x)$.
\end{proof}

By \Cref{lemma: convergence to xl} and \Cref{lemma: distance between two saddle points}, we can directly have the following theorem.

\begin{theorem}[Constant $l(n)\equiv l$ with decay step size $\alpha(n)$]\label{theorem: constant l and decay alpha}
    With the same conditions and assumptions given in \Cref{lemma: convergence to xl} and \Cref{lemma: distance between two saddle points}, the output of \Cref{algorithm} satisfies $\Vert \x^{end}-\x^*\Vert_2\leqslant \Vert\x^{end}-\x^*_l\Vert_2+\Vert\x^*-\x^*_l\Vert_2<\epsilon_\x+\frac{Ml^2d}{\mu_l},$ which means that as the number of iterations tends to infinity, so that the tolerance $\epsilon_\x$ approaches zero, the distance between the output $\x^{end}$ and the target saddle point $\x^*$, i.e. $\Vert \x^{end}-\x^* \Vert_2^2$, is almost surely no greater than a small constant $\left(\frac{Ml^2d}{\mu_l}\right)^2=O(l^4)$.
\end{theorem}

Although the above discussions theoretically show that the iterations of \Cref{algorithm} converge almost surely to 
$\x^*$ (with a decaying $l(n)$) or to $\x_l^*$ (with a constant $l(n)$) conditioned on the boundedness event, the decaying step size defined in \Cref{assumption: step size sec 1} leads to a sublinear convergence rate. From \Cref{Lemma: Almost unbiased estimator}, as the iteration point approaches the saddle point, $\Vert\nabla f(\x)\Vert_2$ vanishes and the variance of $F(\x,\rvec,l)$ nearly goes to zero. This phenomenon is known as variance reduction \cite{johnson2013accelerating,cutkosky2019momentum}, which is commonly used in stochastic gradient methods to accelerate the convergence. When this occurs, it is usually unnecessary to use a decay step size to control the stochastic noise, allowing the use of a constant step size to achieve a linear convergence rate.

\begin{theorem}[Constant $l(n)\equiv l$ with constant step size $\alpha(n)\equiv \alpha$]\label{Theorem: linear convergence rate}
Suppose the \Cref{assumption: regularity}, \Cref{assumption: Hessian}, and \Cref{assumption on theta} hold. Let us pick a step size \(\alpha(n) \equiv \alpha>0\) sufficiently small such that $0<\alpha < \frac{(1-\sqrt{\theta})\mu-M\theta-5M\sqrt{\theta}}{(2d+4)M^2+1}.
$
Then, provided that the boundedness event \(E^\infty\) occurs with a positive probability, the iterations generated by \Cref{algorithm} approach a neighborhood of  \(\x^*\) with radius \(O(l^2/\sqrt{\alpha})\), at a linear decay rate determined by a constant \(0 < \sigma < 1\), i.e.,
\begin{equation}\label{eq: linear rate and plateau error}
\mathbb{E}[\Vert \x(n) - \x^*\Vert_2^2|E^\infty]\leqslant O(\sigma^n+l^4/\alpha).
\end{equation}
\end{theorem}
\begin{proof}
    Define the events $E^n = \{ \x(i) \in U, i=1,2,\cdots,n \}$ and corresponding indicator function $\mathbf{1}_{E^n}$. From \Cref{telescoping with inexact eigenvector} in the Appendix and \Cref{Lemma: Almost unbiased estimator}, we have
$$
\begin{aligned}
&\mathbb{E}[\mathbf{1}_{E^{n+1}}\Vert \x(n+1) - \x^*\Vert_2^2] \leqslant \mathbb{E}[\mathbf{1}_{E^{n}}\Vert \x(n+1) - \x^*\Vert_2^2]\\
&\leqslant \mathbb{E}[\mathbb{E}[ \mathbf{1}_{E^{n}}\Vert \x(n+1) - \x^*\Vert_2^2|\mathcal{F}(n)]]  \\
&\leqslant (1-\alpha((1-\sqrt{\theta})\mu-M\theta-5M\sqrt{\theta})) \mathbb{E}[\mathbf{1}_{E^n} \Vert\x(n) - \x^*\Vert_2^2] \\
&\quad + \alpha M d l^2\mathbb{E}[\mathbf{1}_{E^n} \Vert\x(n) - \x^*\Vert_2] \\
&\quad +\alpha^2 \left((2d+4)\mathbb{E}[\mathbf{1}_{E^{n}} \Vert\nabla f(\x)\Vert^2_2]+{M^2 l^4  d(d+2)(d+4)(d+6)}/{18}\right).
\end{aligned}
$$
From Young's inequality and Jensen's inequality, we have
$$
\begin{aligned}
 \alpha M d l^2\mathbb{E}[\mathbf{1}_{E^n} \Vert\x(n) - \x^*\Vert_2]&\leqslant \alpha^2 (\mathbb{E}[\mathbf{1}_{E^n} \Vert\x(n) - \x^*\Vert_2])^2+{M^2d^2l^4}/{4}\\
 &\leqslant \alpha^2 \mathbb{E}[\mathbf{1}_{E^n} \Vert\x(n) - \x^*\Vert_2^2]+{M^2d^2l^4}/{4}.
\end{aligned}
$$
Note that $\mathbf{1}_{E^n}\Vert\nabla f(\x(n))\Vert_2^2\leqslant\mathbf{1}_{E^n} M^2\Vert\x(n) - \x^*\Vert_2^2$. Then, we get 
$$
\mathbb{E}[\mathbf{1}_{E^{n+1}}\Vert \x(n+1) - \x^*\Vert_2^2] 
\leqslant 
\sigma \mathbb{E}[\mathbf{1}_{E^n} \Vert\x(n) - \x^*\Vert_2^2] + C_4,
$$
where 
$\sigma=\sigma(\alpha,\theta,\mu,M,d)=1-\alpha((1-\sqrt{\theta})\mu-M\theta-5M\sqrt{\theta}))+\alpha^2(2d+4)M^2+\alpha^2$
and $C_4=\frac{M^2d^2l^4}{4}+\frac{\alpha^2M^2 l^4 d(d+2)(d+4)(d+6)}{18}=O(l^4)$.
By telescoping, we get
$$
\begin{aligned}
&\mathbb{E}[\mathbf{1}_{E^{n}}\Vert \x(n) - \x^*\Vert_2^2]\leqslant \sigma^n \Vert\x(0)-\x^* \Vert_2^2+C_4\sum_{i=0}^{n-1}\sigma^i\leqslant \sigma^n \Vert\x(0)-\x^* \Vert_2^2+
\frac{C_4}{1-\sigma}.
\end{aligned}
$$
Since $C_4 =O(l^4)$ and $1-\sigma=O(\alpha)$, this completes the proof.
\end{proof}

From \eqref{eq: linear rate and plateau error}, we can expect that the performance of \Cref{algorithm} initially exhibits a linear convergence rate $\sigma^n$, and then levels off at a small error of $O(l^4/\alpha)$, which leads to the plateau error as commonly referred to in optimization.
\begin{remark}
The numerical tests in \Cref{sec: numerical results} indicate that the boundedness assumption in \Cref{Theorem: linear convergence rate} indeed holds. If suitable global information of the objective function is known, the boundedness assumption can be removed. For instance, if the eigenvalues of \(\nabla^2 f(\x)\) satisfy \eqref{eq:eigenvalues} and
$$\|\nabla^2 f(\x) - \nabla^2 f(\mathbf{y})\|_2 \leqslant C_5 < (1-\sqrt{\theta})\mu - M\theta - 5M\sqrt{\theta}, \forall \x, \mathbf{y} \in \mathbb{R}^d,
$$
the attraction basin of the saddle-search algorithm can be proved to be the entire space \(\mathbb{R}^d\). Suppose both $\alpha$ and $\theta$ are sufficiently small, then $\sigma\approx 1-\alpha\mu$. Thus, a larger $\alpha$ not only accelerates the error decay but also reduces the plateau error of the order $O(l^4/\alpha)$, see the numerical experiments presented in \Cref{sec: MB} for further illustrations of this point. Therefore, \Cref{Theorem: linear convergence rate} suggests choosing the step size $\alpha$ as large as possible while maintaining numerical stability.
\end{remark}

\section{Numerical experiments}\label{sec: numerical results}

In this section, we present a series of numerical examples to demonstrate the practical applicability and the convergence rate of the algorithm, as well as how the plateau error depends on $l$ and $\alpha$.

Note that for $\x(n)$ near a strict saddle point $\x^*$, we have
$$
    \mu^2\Vert \x(n)-\x^* \Vert^2_2\leqslant \Vert\nabla f(\x(n))\Vert^2_2=\Vert\nabla f(\x(n))-\nabla f(\x^*)\Vert^2_2\leqslant M^2\Vert \x(n)-\x^* \Vert^2_2.
$$
Thus, both $\Vert \x(n)-\x^* \Vert^2_2$ and $\Vert\nabla f(\x(n))\Vert^2_2$ can measure the error.

\subsection{Müller-Brown potential}\label{sec: MB}

This numerical example illustrates the practicality of the proposed derivative-free saddle-search algorithm and verifies the predicted order of the plateau error established in the theoretical analysis. The Müller-Brown (MB) potential is a two-dimensional function in theoretical chemistry that describes a system with multiple local minima and saddle points, serving as a standard benchmark to test saddle-search algorithms. The MB potential is given by \cite{bonfanti2017methods}
$$
    E(x,y)=\sum_{i=1}^4A_i e^{a_i(x-\bar{x}_i)^2+b_i(x-\bar{x}_i)(y-\bar{y}_i)+c_i(y-\bar{y}_i)^2},
$$
where $A=[-200,-100,-170,15]$, $a=[-1,-1,-6.5,0.7]$, $b=[0,0,11,0.6]$, $c=[-10,-10,-6.5,0.7]$, $\bar{x}=[1,0,-0.5,-1]$, and $\bar{y}=[0,0.5,1.5,1].$ We show the contour plot of the MB potential in \Cref{fig: MB}(a) with two local minimizers (the reactant and product) and one index-1 saddle point (transition state) connecting two minimizers. 

We set the initial condition to $\x(0)=(0,1)^\top$, and compute the error versus the iteration number for \(l = 0.001\) (\Cref{fig: MB}(b)) and \(l = 0.0001\) (\Cref{fig: MB}(c)), respectively. The results show that at the early stage of the iterations, the error decreases at an exponential rate, and eventually levels off at a small plateau error, which is consistent with \Cref{Theorem: linear convergence rate}. Moreover, as \(l\) decreases from 0.001 to 0.0001, the plateau error decreases from \(10^{-11}\) to \(10^{-15}\), which aligns with the \(O(l^4)\) error predicted by both \Cref{theorem: constant l and decay alpha} and \Cref{Theorem: linear convergence rate}. To further verify the \(O(l^4/\alpha)\) plateau error in \Cref{Theorem: linear convergence rate}, we examine the plateau error for different difference lengths \(l\) and step sizes \(\alpha\) in \Cref{Table:Muller}. The plateau error is calculated as the average of the minimum errors over $100$ runs, defined by \(\frac{1}{100}\sum_{i=1}^{100} \min_n \Vert \x^i(n) - \x^* \Vert_2^2\). With respect to \(l\), the rate of decay of the plateau error is approximately $O(l^4)$. Regarding \(\alpha\), when \(\alpha\) increases from 0.0001 to 0.0002, the plateau error roughly halves. The numerical tests in \Cref{Table:Muller} are consistent with the \(O(l^4/\alpha)\) scaling.
\begin{figure}
    \centering
    \includegraphics[width=0.99\linewidth]{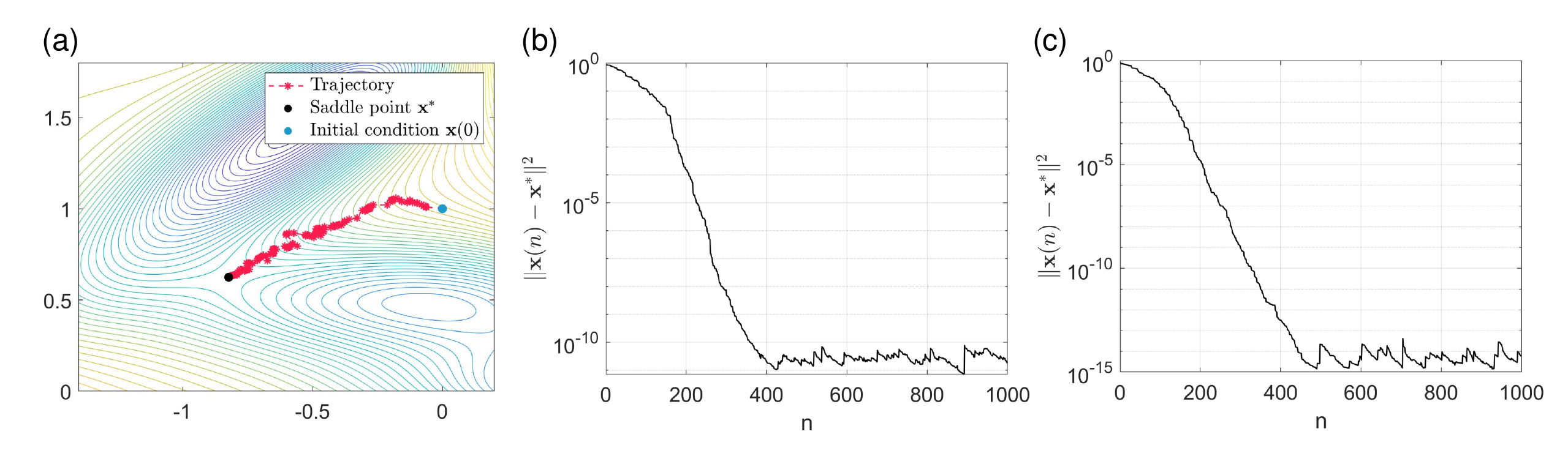}\vspace{-.3cm}
    \caption{Contour plot of MB potential and iteration points (a) of one-time run of \Cref{algorithm}. 
    Error plots of the iteration points 
    with (b) $l=0.001, \alpha_{\x}\equiv 
0.0001,n_\x^{max}=1000, n_\vvec^{max}=100$, $\alpha_\vvec\equiv0.0002$ and (c) $l=0.0001, \alpha_\x\equiv 
0.0001,n_\x^{max}=1000, n_\vvec^{max}=100$, $\alpha_\vvec\equiv0.0002/d$.}
    \label{fig: MB}\vspace{-.3cm}
\end{figure}

\begin{table}[h]\label{Table:Muller}
\centering
\caption{Plateau error is computed as the mean of the minimum errors across $k=100$ runs, defined by $\frac{1}{k}\sum_{i=1}^k\min_n \Vert \x^i(n)-\x^* \Vert_2^2$, with $n_\x^{max}=1000, n_\vvec^{max}=100$, $\alpha_\vvec\equiv0.0002, l(n)\equiv0.001$.}\vspace{-.2cm}
\begin{tabular}{c|cc|cc}
\hline
    & \multicolumn{2}{c|}{$\alpha=0.0001$} & 
    \multicolumn{2}{c}{$\alpha=0.0002$}  \\
    \hline
    l &Plateau error & Order of vanishing  & Plateau error & Order of vanishing  
 \\
    \hline
     $2^{-8}$ & 2.71E-09  &  & 1.28E-09 &    \\
 
    $2^{-9}$ &1.58E-10 & 4.28 & 7.73E-11 & 4.15   \\

    $2^{-10}$ &1.02E-11 & 3.89 & 4.84E-12 & 3.99  \\

    $2^{-11}$ &6.40E-13 & 3.97 & 2.96E-13 & 4.09  \\

    $2^{-12}$ &3.87E-14 & 4.14 & 2.02E-14 & 3.66  \\
    \hline 
\end{tabular}
\vspace{-.3cm}
\end{table}

\subsection{Implicitly defined objective function}
This numerical example demonstrates that our method is applicable to implicitly defined objective functions, whose gradients and Hessians are difficult to obtain explicitly. Such problems arise, for instance, in hyperparameter optimization \cite{franceschi2018bilevel} and in self-consistent field theory in materials science \cite{fredrickson2006equilibrium}, where the objective function is typically defined through a minimization problem. For brevity and clearer visualization, we consider the following two-dimensional implicit function:
\begin{equation}\label{eq: implicit function}
f(x,y)=\min_{\mathbf{z}} g(x,y,\mathbf{z})=\min_{\mathbf{z}} (x-z_1)^2+(y-z_2)^2+\sin(z_1z_2),
\end{equation}
where each function evaluation is defined as the solution of an associated optimization problem. In \Cref{fig: implicit example}(a), we present the contour plot of the energy landscape, which indicates that $\x^*=(0,0)$ is an index-1 saddle point of $f(x,y)$.

Actually, the Hessian of $f$ is given by the implicit function theorem
$$
\nabla^2 f(x,y)=
\begin{pmatrix}
g_{xx} & g_{xy}\\
g_{yx} & g_{yy}
\end{pmatrix}
-
\begin{pmatrix}
g_{x z_1} & g_{x z_2}\\
g_{y z_1} & g_{y z_2}
\end{pmatrix}
\begin{pmatrix}
g_{z_1 z_1} & g_{z_1 z_2}\\
g_{z_2 z_1} & g_{z_2 z_2}
\end{pmatrix}^{-1}
\begin{pmatrix}
g_{z_1 x} & g_{z_1 y}\\
g_{z_2 x} & g_{z_2 y}
\end{pmatrix}.
$$
With all derivatives are evaluated at $(x,y,\mathbf{z})=(0,0,\mathbf{0})$, the Hessian of $f$ at $\x^*$ is
\[
\nabla^2 f(0,0) = \begin{bmatrix}2 & 0 \\ 0 & 2 \end{bmatrix} - \begin{bmatrix}8/3 & -4/3 \\ -4/3 & 8/3\end{bmatrix} = \begin{bmatrix}-2/3 & 4/3 \\ 4/3 & -2/3\end{bmatrix},
\]
whose eigenvalues are $\lambda_1 = 2/3, \lambda_2 = -2,$
showing that the origin is an index-1 saddle point of $f(x,y)$. In \Cref{fig: implicit example}(b), we plot the error versus the iteration number, showing that the iterates approach the saddle point with a linear convergence rate before leveling off at a small plateau error.

\begin{figure}
    \centering
\includegraphics[width=0.75\linewidth]{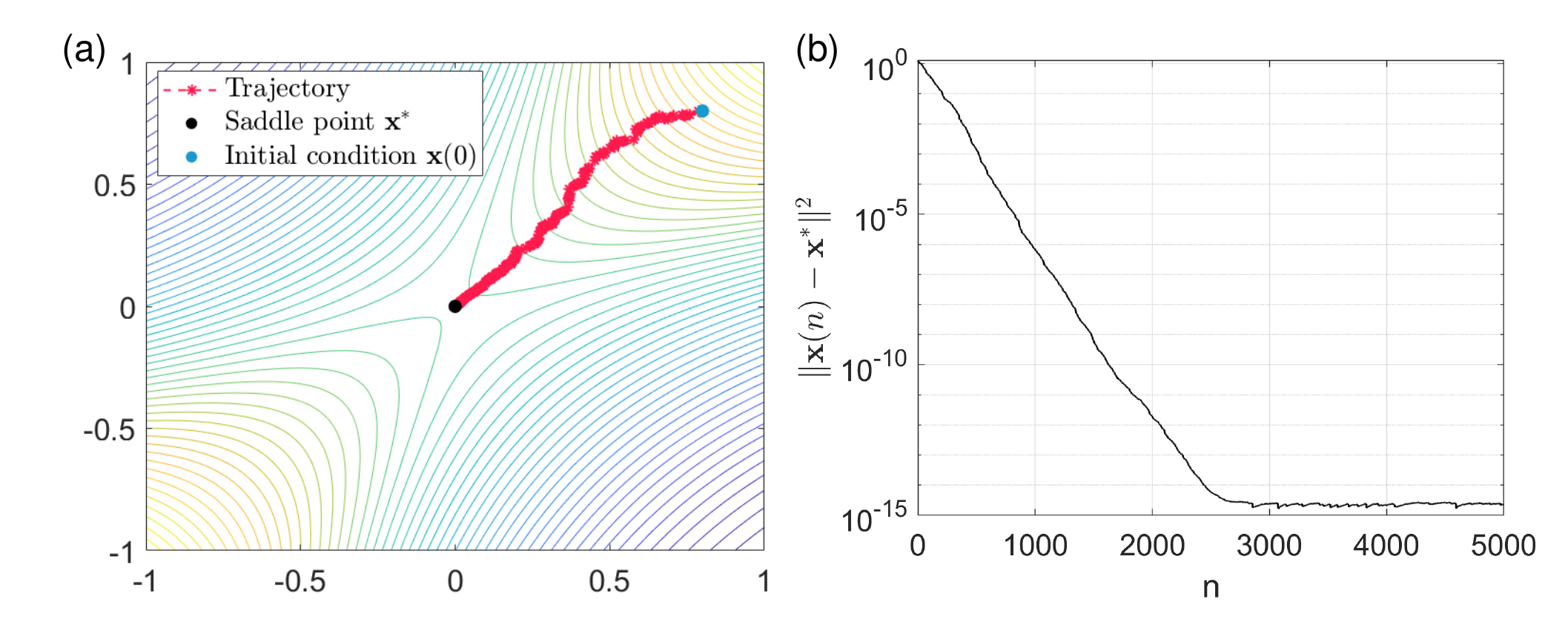}\vspace{-.4cm}
    \caption{(a) Contour plot of \eqref{eq: implicit function} and iteration points of one-time run of \Cref{algorithm}. (b) Error plot with $l=0.1, \alpha_\x\equiv 
0.01,n_\x^{max}=5000, n_\vvec^{max}=100$, $\alpha_\vvec\equiv0.0002/d$.}
    \label{fig: implicit example}
    \vspace{-.3cm}
\end{figure}

\subsection{Modified Rosenbrock function}\label{sec: Rosenbrock}
The following numerical experiments demonstrate that the algorithm can be applied to locate the high-index saddle point of the high-dimensional models, show the efficiency of the Hessian-vector product approximation compared with the Hessian approximation, and show that the zeroth-order algorithm also encounters slow convergence rates on ill-conditioned problems, similar to the typical gradient-based algorithms. The Rosenbrock function \cite{shang2006note}, also known as the banana function, features a narrow and curved valley with a flat energy landscape around the minimum.
It is a classical benchmark for optimization algorithms due to its challenging geometry, which makes convergence to the global minimum difficult. To adapt the Rosenbrock example for saddle-point computation, we modify it by adding extra quadratic arctangent terms \cite{2019High}, yielding
$$
  f(\x)=\sum_{i=1}^{d-1}100(x_{i+1}-x_i^2)^2+(1-x_i)^2+\sum_{i=1}^ds_i\arctan^2(x_i-1),
$$
where $\{s_i\}_1^d$ are the parameters introduced to tune the index of the critical point $\x^*=(1,\cdots,1)^\top$. For example, when $d=2, s_1=-50,s_2=1$, $\x^*$ is an index-1 saddle point, and the condition number $Cond(\nabla^2f(\x^*))\approx47$. The energy landscape of the modified Rosenbrock function, shown in \Cref{fig: 2D RB}(a), illustrates that the energy landscape near $\x^*$ is relatively flat along the unstable direction. The convergence rate largely depends on the condition number of $\nabla^2f(\x^*)$. According to \Cref{Theorem: linear convergence rate}, the error decays approximately as $\sigma^n\approx (1-\alpha \mu)^n$. For numerical stability, the step size $\alpha$ should be smaller than $2/M$, which yields an error decay rate on the order of $(1-2/Cond(\nabla^2f(\x^*)))^n$. It can be further observed from the iteration trajectory in \Cref{fig: 2D RB}(a) that the iteration points tend to converge more rapidly along the steep directions, while moving slowly along the eigenvector corresponding to the smallest absolute eigenvalue. It is reflected in \Cref{fig: 2D RB}(b) that the error decreases rapidly at the beginning of the iteration (along the steep directions), but slows down in the later stage (along the flat directions).

We then examine the advantage of variance reduction when using the Hessian-vector approximation $H_\vvec$, as compared with the full Hessian approximation $\mathbf{H}$, in the eigenvector search. We search for the smallest eigenvalue and its corresponding eigenvector of $\nabla^2 f(\x^* + 0.5 \rvec)$, where $\rvec$ is drawn from the standard normal distribution, and plot the resulting error in \Cref{fig: 1000D RB}(a). When the dimension is low, i.e., $d=3$, both the Hessian approximation and the Hessian-vector approximation perform well. However, for $d=100$, the variance of the Hessian approximation $\mathbf{H}$ reaches the order of $10^8$, whereas that of the Hessian-vector approximation is only about $10^2$. As a result, the eigenvector search using the Hessian approximation becomes unstable and fails to converge to the true eigenvector $\vvec^*$, while the Hessian-vector approximation still performs well. Although thousands of iterations are required to obtain an accurate output in \Cref{fig: 1000D RB}(a), in practice, using the approximate unstable directions of $\nabla^2 f(\x(n))$ as the initial condition, we typically find the approximate unstable directions of $\nabla^2 f(\x(n+1))$ within tens of iterations in \Cref{algorithm vector}. This is because $\x(n)$ is close to $\x(n+1)$, so the unstable directions of $\nabla^2 f(\x(n))$ (which we already have) serve as good initial conditions for those of $\nabla^2 f(\x(n+1))$.

We now consider a high-index saddle point in a high-dimensional setting. We choose $d=1000$ and set $s_1 = s_2 = s_3 = -1000$, $s_i = 1$ for $i=4,\dots,1000$. Then, $\x^* = (1,\dots,1)^\top$ is an index-3 saddle point with a large condition number, $Cond(\nabla^2 f(\x^*)) \approx 722$. In \Cref{fig: 1000D RB}(b), we plot the error versus iteration for both the zeroth-order algorithm and the deterministic algorithm. The error initially decays rapidly along the steepest directions and then slows down, similar to the behavior observed in \Cref{fig: 2D RB}, but with a slower convergence rate due to the larger condition number of the Hessian. Hence, an interesting and important future direction could be to accelerate the convergence by introducing momentum to address the ill-conditioning. Importantly, the zeroth-order algorithm achieves a convergence rate comparable to that of the deterministic algorithm, without requiring gradient and Hessian evaluations.

\begin{figure}
    \centering    \includegraphics[width=0.8\linewidth]{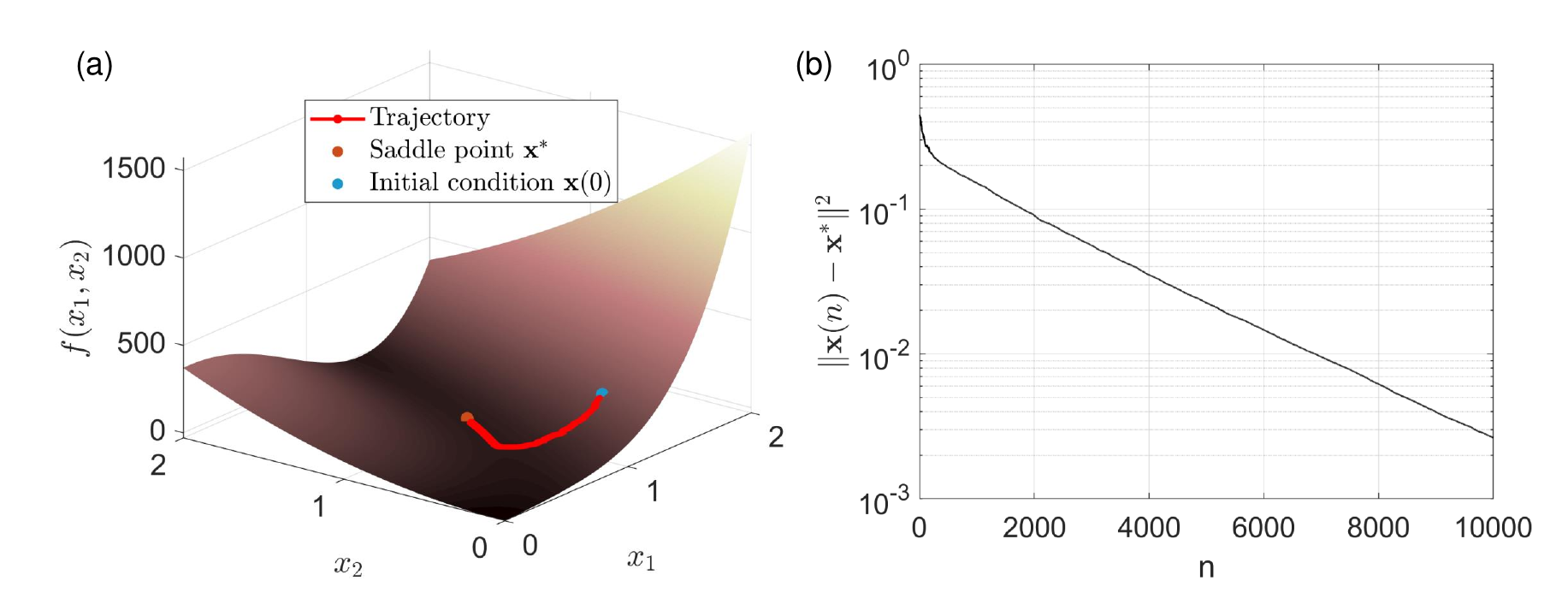}\vspace{-.3cm}
    \caption{(a) Energy landscape of the 2D Rosenbrock function and iteration points of one-time run of \Cref{algorithm}. (b) Error plot of the iteration points of \Cref{algorithm} with $l=0.0001, \alpha_\x\equiv 
0.00001,n_\x^{max}=10000, n_\vvec^{max}=100$, $\alpha_\vvec\equiv0.0002/d$.}\vspace{-.4cm}
    \label{fig: 2D RB}
\end{figure}

\begin{figure}
    \centering
\includegraphics[width=0.8\linewidth]{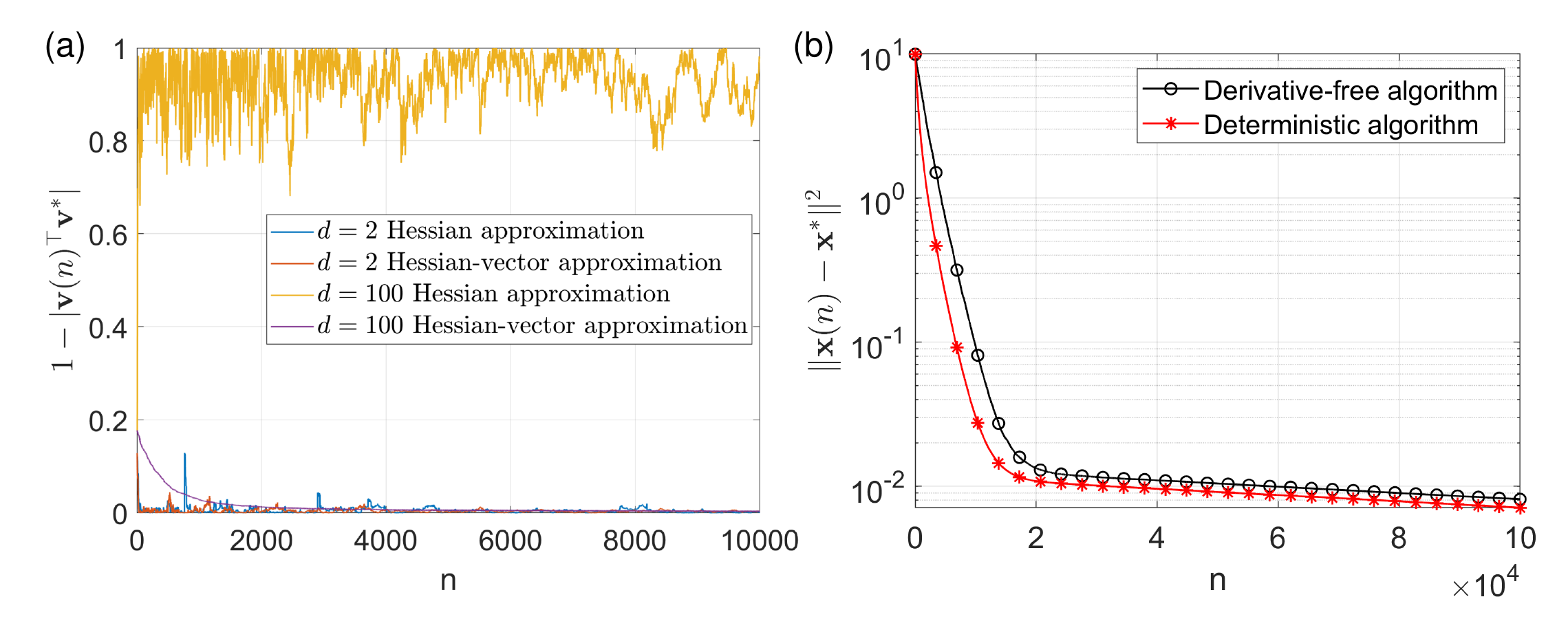}\vspace{-.4cm}
    \caption{(a)The error plots of the inner eigenvector search algorithm (\Cref{algorithm vector}) with $d=2$ and $d=100$, using the Hessian approximation ($\mathbf{H}\vvec$) and the Hessian–vector approximation ($H_{\vvec}$), respectively. The step size is choosen as $\alpha(n) = \frac{1}{d(10n+10000)}$. (b) Error plots of \Cref{algorithm} and the deterministic saddle-search algorithm in \eqref{eq: saddle algorithm} applied to the modified Rosenbrock function with $d=1000,$ $l=0.0001, \alpha_\x\equiv 
0.000001,n_\x^{max}=100000, n_\vvec^{max}=100$, $\alpha_\vvec\equiv0.0002/d$.}
    \label{fig: 1000D RB}\vspace{-.1cm}
\end{figure}

\subsection{Loss landscape of a neural network}
In neural networks, the model parameters are sometimes hidden, and accessing gradient information is inadmissible, as in zeroth-order attacks \cite{liu2018zeroth}. In such cases, one typically has to resort to the derivative-free methods. We implement the derivative-free saddle-search algorithm on the loss landscape of the linear neural network to show that it also works for the degenerate high-index saddle point. Because of the overparameterization of neural networks, most critical points of the loss function are highly degenerate, i.e., the Hessian contains many zero eigenvalues. We consider a fully-connected linear neural network $   Net(\W,\x)=\W_H\W_{H-1}\cdots \W_2 \W_1 
 \x$
 of depth $H$, with weight parameters $\W_{i+1}\in \mathbb{R}^{d_{i+1}\times d_i}$ for $d_0=d_x, d_H=d_y$. The corresponding empirical loss $f$ is defined by
$$    f(\W)=\sum_{i=1}^N\Vert Net(\W,\x_i)-\y_i \Vert_2^2=\Vert \W_H\W_{H-1}\cdots \W_2 \W_1\mathbf{X}-\mathbf{Y} \Vert_F^2,
$$
where $\mathbf{X}=[\x_1,\cdots,\x_N]$,  and $\mathbf{Y}=[\y_1,\cdots,\y_N]$, with $\{(\x_i,\y_i)\}_{i=1}^{N}$ being the training data.
If $d_i = d_0, 1\leqslant i \leqslant H-1$, then $\W^*=[\W_1^*,\cdots,\W_H^*]$ is a saddle point where
$$
\W_1^* = \left[ \begin{array}{cc}
U_{\mathcal{S}}^\top \Sigma_{YX} \Sigma_{XX}^{-1} \\
\mathbf{0}
\end{array} \right],  
\W_h^* = \I_{d_0} \text{ for } 2 \leqslant h \leqslant H - 1, \text{ and } 
\W_H^* = \left[ \begin{array}{cc}
\mathbf{U}_{\mathcal{S}} , \mathbf{0}
\end{array} \right], 
$$
and $\mathcal{S}$ is an index subset of $\{1,2,\ldots,r_{\max}\}$ for $r_{\max} = \min\{d_0,\ldots,d_H\}$, 
$\Sigma_{XX} = XX^\top$, $\Sigma_{YX} = YX^\top$, $\Sigma = \Sigma_{YX} \Sigma_{XX}^{-1} \Sigma_{YX}^\top$, and $U$ satisfies 
$\Sigma = U \Lambda U^\top$ with $\Lambda = \mathrm{diag}(\lambda_1, \ldots, \lambda_{d_y})$ \cite{achour2024loss,luo2025accelerated}. 

\begin{figure}
    \centering    \includegraphics[width=0.36\linewidth]{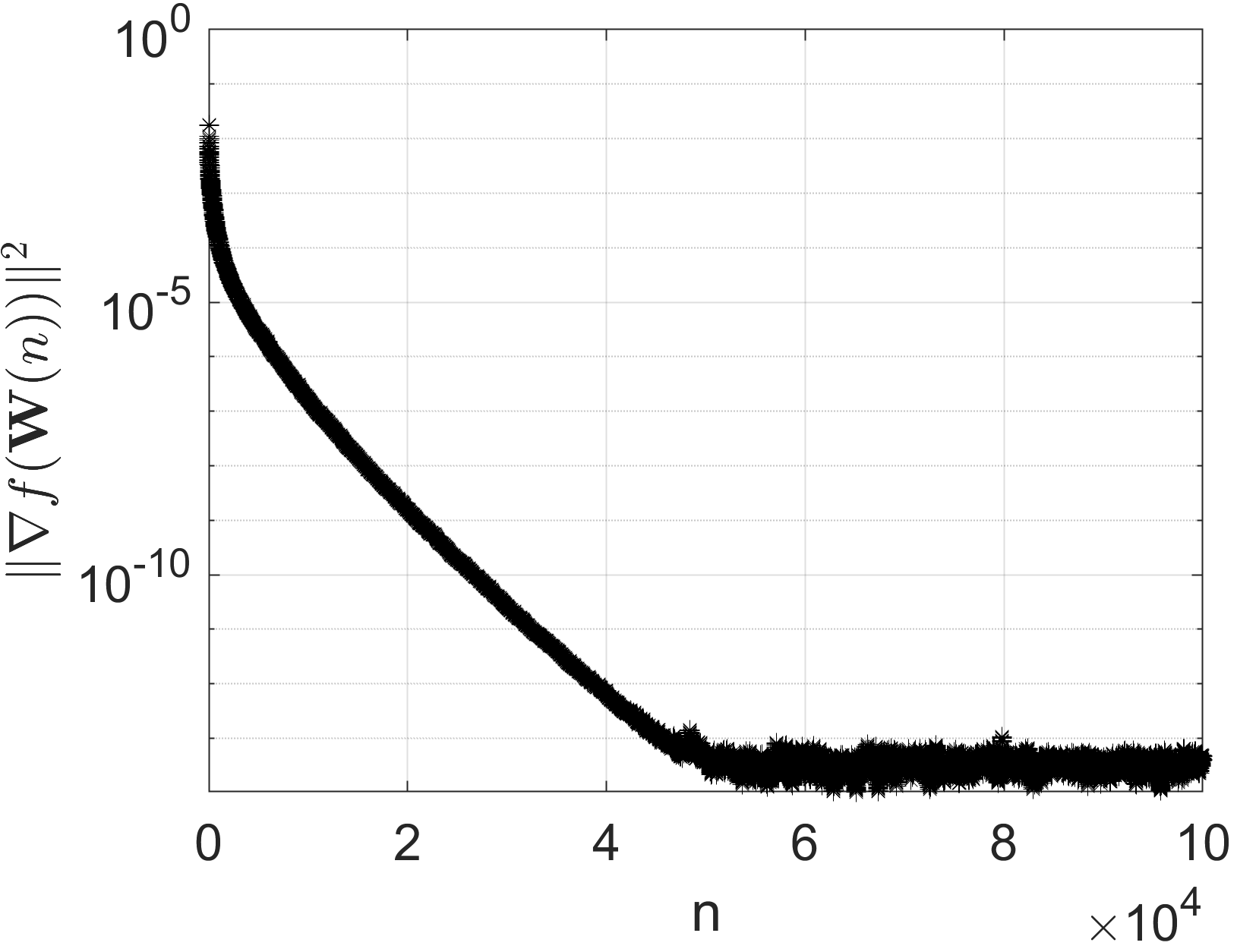}\vspace{-.3cm}
    \caption{Plots of the squared gradient norm $\|\nabla f(\mathbf{W}(n))\|_2^2$, where $\mathbf{W}(n)$ is the iteration points of \Cref{algorithm} with $\alpha_\x\equiv0.01, l= 0.0001$, for the neural network loss function.}\vspace{-.4cm}
    \label{fig: NN}
\end{figure}

We set the depth $H = 5$, the input dimension $d_x = 10$, the output dimension $d_y = 4$, $d_i = 10$ for 
$1 \leqslant i \leqslant 4$, and the number of data points $N=100$. Data points $(\x_i, \y_i)$ are sampled from the standard normal distributions. Under the current setting, $\W^*$ 
is a degenerate saddle point with $16$ negative eigenvalues and several zero eigenvalues. In \Cref{fig: NN}, we tested the derivative-free saddle-search algorithm, and the algorithm successfully locates the target saddle point, with a linear convergence rate. The error still levels off at a small plateau value, which is about $10^{-14}$, and this plateau error can be further reduced by decreasing $l$.

\section{Conclusion and discussion}\label{sec: conclusion}

This work focuses on the development and numerical analysis of a derivative-free algorithm for saddle point search, which requires only function evaluations. The proposed derivative-free saddle-search algorithm consists of two main components: an inner eigenvector-search step and an outer saddle-search step. For the inner eigenvector-search step, we adopt a zeroth-order Hessian-vector approximation to eliminate the curse of dimensionality, and prove that it almost surely identifies approximate unstable directions of both $\nabla^2 f(\x)$ and $\nabla^2 f_l(\x)$ within a finite number of iterations. Using these approximate unstable directions, we further prove that the outer derivative-free saddle-search step can identify the saddle point $\x^*$ under a decaying difference length $l(n)$ and step size $\alpha(n)$. With a constant difference length and decay step size, the outer iteration almost surely converges to $\x_l^*$, which is a saddle point of the Gaussian smoothed function and is $O(l^2)$-close to $\x^*$. Moreover, under a boundedness assumption, we establish that with both constant difference length and constant step size, the iterates approach a neighborhood of $\x^*$ at a linear convergence rate, and the radius of this neighborhood is $O(l^2/\sqrt{\alpha})$. These results demonstrate that when the difference length $l$ is sufficiently small, the algorithm output can closely approximate the target saddle point, which is further confirmed in our numerical experiments.

This study opens several potential directions for further exploration. The performance of both the derivative-free eigenvector-search algorithm and the derivative-free saddle-search algorithm can be improved by incorporating certain techniques from derivative-free optimization methods. For example, the gradient of the Gaussian-smoothed function is difficult to estimate because it involves a $d$-dimensional integral; a directional Gaussian smoothing may provide a more efficient way to design the gradient estimator \cite{tran2025convergence}. Moreover, by introducing the error-correction term, the variance of the estimator can be reduced \cite{liu2018zeroth,ji2019improved}, thereby accelerating convergence as stated in \eqref{convergence rate versus variation}.
In addition, incorporating momentum into the algorithm, as in \cite{chen2019zo}, may accelerate convergence, especially when the energy landscape near the saddle point $\x^*$ is flat (or $\nabla^2f(\x^*)$ is ill-conditioned), such as in the modified Rosenbrock function in \Cref{sec: Rosenbrock}. A possible theoretical challenge, however, lies in ensuring that the iterates remain within the neighborhood of the target saddle point, at least with positive probability. Overall, since the derivative-free saddle-search algorithm could be seen as a gradient-estimate method, it is reasonable to expect that many techniques developed for the derivative-free optimization methods could be applied to it as well. However, the additional difficulties arise from the intrinsically unstable nature of saddle points and the nonconvexity of the objective function.

\section*{A. Appendix}
\subsection*{A.1. Proof of \Cref{Lemma: Almost unbiased estimator}}
\begin{proof}
Denote $\rvec^{\otimes i}=[\underbrace{\rvec,\cdots,\rvec}_{i\ \text{times}}]:=\underbrace{\rvec \otimes \rvec \otimes \cdots \otimes \rvec }_{i\ \text{times}}$, for $ i\geqslant 2.$ We have
$$
\begin{aligned}
\nabla f_l(\x)&=\mathbb{E}_\rvec[\nabla f(\x+l\rvec)]=\nabla f(\x)+{l^2}\mathbb{E}_\rvec[\nabla^3 f(\zeta_{1,l,\rvec})[\rvec,\rvec]]/2,\\
\nabla^2 f_l(\x)&=\mathbb{E}_\rvec[\nabla^2 f(\x+l\rvec)]=\nabla^2f(\x)+{l^2}\mathbb{E}_\rvec[\nabla^4 f(\zeta_{2,l,\rvec})[\rvec,\rvec]]/2,
\end{aligned}
$$
from the Taylor expansion. Hence
$$
\|\nabla f_l(\x)-\nabla f(\x)\|_2
\leqslant {Ml^2d}/{2},\; \|\nabla^2 f_l(\x)-\nabla^2 f(\x)\|_2
\leqslant {Ml^2d}/{2}.
$$
To calculate the second-order moment of $F(\x,\rvec,l)$, we consider
$$ 
   f(\x\pm l\rvec)=f(\x)\pm l\nabla f(\x)^\top\rvec+{l^2}\rvec^\top\nabla^2f(\x)\rvec\/2 \pm {l^3}\nabla^3f(\xi_\pm)[\rvec,\rvec,\rvec]/6,
   $$
Then, $F(\x,\rvec,l)=(\nabla f(\x)^\top\rvec)\rvec+M_2\rvec$, where $|M_2|\leqslant\frac{l^2M\Vert\rvec\Vert^3_2}{6}.$ From Isserlis's theorem,
$$
\mathbb{E}_\rvec(\Vert(\nabla f(\x)^\top\rvec)\rvec\Vert^2)=\sum_{i=1}^d\sum_{j=1}^d\sum_{k=1}^d\nabla f(\x)_j \nabla f(\x)_k \mathbb{E}( r_i r_i r_j r_k)=(d+2)\Vert\nabla f(\x)\Vert^2,
$$
   $$
\begin{aligned}
\mathbb{E}_{\mathbf{r}}(\Vert F(\x,\mathbf{r},l)\Vert^2)\leqslant2\mathbb{E}_\rvec(\Vert(\nabla f(\x)^\top\rvec)\rvec\Vert^2_2)+{M^2l^4}\mathbb{E}_\rvec(\Vert\rvec\Vert^8)/18\\
\leqslant(2d+4)\Vert\nabla f(\x)\Vert^2_2+{M^2l^4 d(d+2)(d+4)(d+6)}/{18}.
\end{aligned}
   $$ 
For the Hessian approximation, we consider the following Taylor expansions,
   $$
       f(\x\pm l\rvec)=f(\x)\pm l\nabla f(\x)^\top\rvec+\frac{l^2}{2}\rvec^\top\nabla^2f(\x)\rvec\pm\frac{l^3}{6}\nabla^3f(\x)[\rvec^{\otimes3}]+\frac{l^4}{24}\nabla^4f(\xi_\pm)[\rvec^{\otimes4}].
   $$
Then $\Hvec(\x,\rvec,l)=\frac{(\rvec^\top\nabla^2 f(\x)\rvec)(\rvec\rvec^\top-\I)}{2}+M_3(\rvec\rvec^\top-\I)$ with $|M_3|\leqslant \frac{l^2M\Vert\rvec\Vert_2^4}{24}.$ 
From
$$
   \begin{aligned}
\mathbb{E}_\rvec[\Vert\rvec^\top \nabla^2f(\x) \rvec(\rvec\rvec^\top-\I)\Vert_2^2/4]
&\leqslant \|\nabla^2f(\x)\|_2^2(\mathbb{E}_\rvec[\|\rvec\|^8+2\|\rvec\|^6+\|\rvec\|^4])/4\\
&\leqslant {d(d+2)(d^2+12d+33)M^2}/{4},
\end{aligned}
$$
we have that
$$
\begin{aligned}
\mathbb{E}_\rvec[\Vert\Hvec(\x,\rvec,l)\Vert_2^2]&\leqslant 2\mathbb{E}_\rvec[\Vert\rvec^\top \nabla^2f(\x) \rvec(\rvec\rvec^\top-\I)\Vert_2^2/4]+2\mathbb{E}_\rvec[M_3^2\Vert\rvec\rvec^\top-\I\Vert_2^2]\\
&\leqslant {d(d+2)(d^2+12d+33)M^2}/{2}\\
&\quad+{M^2 l^4 d(d+2)(d+4)(d+6)(d^2+20d+97)}/{288}.
\end{aligned}
$$
\end{proof}

\subsection*{A.2. Proof of \Cref{Hv expectation and variation}}
\begin{proof} 
$$
\begin{aligned}
\mathbb{E}_\rvec[H_\vvec]&=\frac{\mathbb{E}_\rvec[F(\x+l\vvec,\rvec,l)]-\mathbb{E}_\rvec[F(\x-l\vvec,\rvec,l)]}{2l}=\frac{\nabla f_l (\x+l\vvec)-\nabla f_l(\x-l\vvec)}{2l}\\
&=\nabla^2f_l(\x)\vvec+\frac{l^2}{12}(\nabla^4 f_l(\xi_+)-\nabla^4 f_l(\xi_-))[\vvec,\vvec,\vvec],
\end{aligned}
$$
$$
\Vert\mathbb{E}_\rvec[H_\vvec]-\nabla^2f(\x)\vvec\Vert_2\leqslant \|\nabla^2f(\x)-\nabla^2f_l(\x)\|_2+{Ml^2}/{6}\leqslant  Ml^2(d/2+6)
\leqslant Ml^2d.
$$
By the Taylor expansions, we get
$$
\begin{aligned}
H_\vvec&=\frac{F(\x+l\vvec,\rvec,l)-F(\x-l\vvec,\rvec,l)}{2l}\\
&=\rvec\rvec^\top\nabla^2f(\x)\vvec+\frac{l^2}{6}(\nabla^4f(\x)[\vvec,\vvec,\vvec,\rvec]+\nabla^4f(\x)[\vvec,\rvec,\rvec,\rvec])\rvec+O(l^4).
\end{aligned}
$$
Thus, from the Cauchy–Schwarz inequality, we have that
$$\mathbb{E}_\rvec (\Vert H_\vvec \Vert_2^2)\leqslant
M^2(4(d+2)+{l^4 d(d+2)}/{9}+{l^4 d(d+2)(d+4)(d+6)}/{9})+O(l^8d^8).$$
\end{proof}

\subsection*{A.3. Proof of \Cref{decay ln and an}}
\begin{lemma}[Proposition 4.15 of \cite{shi2025stochastic}]\label{telescoping with inexact eigenvector}
    Suppose the \Cref{assumption: regularity}, \Cref{assumption: Hessian} and \Cref{assumption on theta} hold. 
    Let a neighborhood $U$ of $\x^*$ be defined as in \eqref{eq:boundness event mu}.

    If $\x(n)\in U$, we almost surely have
\begin{equation}\label{descent term}
    \begin{aligned}
        \Vert\x(n+1)-\x^*\Vert^2_2/2\leqslant& (1-\alpha(n)((1-\sqrt{\theta})\mu-M\theta-5M\sqrt{\theta}))\Vert\x(n)-\x^*\Vert^2_2/2\\
&+\alpha(n)\psi(n)+\frac{1}{2}\alpha(n)^2\Vert F(\x(n),\rvec(n),l(n))\Vert^2_2,
        \end{aligned}
\end{equation}
where $\psi(n) =-\langle (\I-2\tilde{V}(n) \tilde{V}(n)^\top)(F(\x(n),\rvec(n),l(n))-\nabla f (\x(n))), \x(n) -\x^* \rangle$.
\end{lemma}

\begin{lemma}\label{attraction domain}
     Suppose the \Cref{assumption: step size sec 1}, \Cref{assumption: decay difference length}, \Cref{assumption: regularity}, \Cref{assumption: Hessian}, and \Cref{assumption on theta} hold. If we assume that the outer saddle-search algorithm is implemented with the step size satisfying that $\sum_{n=0}^\infty \alpha(n)^2$ is sufficiently small (e.g. $\alpha(n)=\frac{\gamma}{(n+m)^p}$ with a large enough $m$ and $p\in(1/2,1]$). Set $\epsilon$ and $U_0$ as
\begin{equation}
\label{eq: epsilon and U0}
    4\epsilon+2\sqrt{\epsilon}=\min
\left
( 
\frac{(1-\sqrt{\theta})\mu}{M}-\theta-5\sqrt{\theta},\delta \right
)^2, 
   U_0=\left\{\x,\Vert\x-\x^*\Vert_2\leqslant\sqrt{2\epsilon}\right\}\subset U,
   \end{equation}
where $U$ is defined in 
\eqref{eq:boundness event mu}. Then, for $\x(n)$ generated by the saddle-search algorithm with an initial condition $\x(0) \in U_0$, the event $ E^\infty$ defined in \eqref{eq:boundness event mu} occurs with a positive probability. In a particular case, if the saddle-search algorithm is run with a step-size schedule of the form $\alpha(n)=\frac{\gamma}{(n+m)^p},p\in(1/2,1]$ and large enough $m$, then $\mathbb{P}(E^\infty| \x(0) \in U_0) >1-\epsilon_{E^\infty}$, and $1-\epsilon_{E^\infty}\rightarrow1$ as $m\rightarrow \infty$.  
\end{lemma}

\begin{proof}
   Since \Cref{telescoping with inexact eigenvector} 
   is analogous to \cite[Proposition D.1]{mertikopoulos2020almost}, which constitutes a key step in establishing the positive probability of the event $E^\infty$, the proof proceeds along the same lines as \cite[Theorem 4]{mertikopoulos2020almost}. Thus, we omit the details and give the main steps. The places different from the original proof are: (i) the gradient term appears as $(\I-2\tilde{V}(n) \tilde{V}(n)^\top)\nabla f$ rather than $\nabla f$, which does not affect its $L^2$ norm; (ii) the upper bound of $\|\nabla f(\x)\|_2$, $G = \max_{\x \in U}\|\nabla f(\x)\|_2$, is not global but restricted to $\x \in U$, which does not alter the proof steps; and (iii) $\psi(n) =-\langle (\I-2\tilde{V}(n) \tilde{V}(n)^\top)(F(\x(n),\rvec(n),l(n))-\nabla f (\x(n))), \x(n) -\x^* \rangle$ is not a martingale sequence with zero expectation, and we need to estimate the error induced by the biased zeroth-order gradient estimator.
   
   By \Cref{telescoping with inexact eigenvector}, the distance between the iteration point and saddle point almost surely grows at most by $\alpha(n)\psi(n)+\frac{1}{2}\alpha(n)^2\Vert  F(\x(n),\rvec(n),l(n)) \Vert^2_2$ at each step. Define the cumulative mean square error $$R^n=\left(\sum_{i=0}^n \alpha(n)\psi(n)\right)^2+\sum_{i=0}^n\frac{1}{2}\alpha(n)^2\Vert F(\x(n),\rvec(n),l(n)) \Vert^2_2.$$
Set $\epsilon$ and $U_0\subset U$ as in \eqref{eq: epsilon and U0}. Then, if $\x(0)\in U_0$, and the cumulative mean
square error $R^n<\epsilon$, we almost surely have that $$
\begin{aligned}
\Vert \x(n+1)-\x^*\Vert^2_2&\leqslant \Vert \x(0)-\x^* \Vert_2^2+2\sum_{i=0}^n \alpha(n)\psi(n)+\sum_{i=0}^n\alpha(n)^2\Vert F(\x(n),\rvec(n),l(n)) \Vert^2_2\\
&\leqslant 2\epsilon+2\sqrt{\epsilon}+2\epsilon=4\epsilon+2\sqrt{\epsilon},
\end{aligned}
$$
i.e., $\x(n+1)\in U.$
Define the small error event $J^n=\{R_k\leqslant\epsilon,k=1,\cdots,n\}$, and the boundedness event $
        E^n = \{ \x(i) \in U, i=1,2,\cdots,n \}.$
By \cite[Lemma D.2.]{mertikopoulos2020almost}, we have $E^{n+1}\subseteq E^n, J^{n+1}\subseteq J^n, J^{n-1}\subseteq E^n$, which means that a small error ensures the occurrence of the boundedness event. Different from \cite[Lemma D.2.]{mertikopoulos2020almost} where an unbiased stochastic gradient is used so that $\psi(n)$ is a martingale with zero expectation, $\psi(n) =-\langle (\I-2\tilde{V}(n) \tilde{V}(n)^\top)(F(\x(n),\rvec(n),l(n))-\nabla f (\x(n))), \x(n) -\x^* \rangle$ is not a martingale sequence with zero expectation. Nevertheless, we can control this term from $l(n)\leqslant L\alpha(n)$ in \Cref{assumption: decay difference length}. Specifically, from \Cref{Lemma: Almost unbiased estimator}, if $\x(n) \in U$,
$$ \Vert\mathbb{E}_{\mathbf{r}}[F(\x(n),\mathbf{r},l(n))]-\nabla f(\x(n))\Vert_2\leqslant {Ml(n)^2d}/{2},$$  $$\mathbb{E}_{\mathbf{r}}[\Vert F(\x(n),\mathbf{r},l(n))\Vert_2^2]\leqslant(2d+4) G^2+{M^2l(n)^4 d(d+2)(d+4)(d+6)}/{18}.$$
By replacing equations D.24a and D.25 in \cite{mertikopoulos2020almost} with 
$$
2\alpha(n+1)\mathbb{E}\left[\left(\sum_{i=0}^{n} \alpha(i)\psi(i)\right)\psi(n+1)\mathbf{1}_{J^{n}}\right]\leqslant \sqrt{\epsilon}\sqrt{4\epsilon+2\sqrt{\epsilon}}ML^2 \alpha(n+1)^2 d,
$$
and following the same steps in \cite[Lemma D.2]{mertikopoulos2020almost}, we have
\begin{equation}\label{telescoping attration}
\epsilon \mathbb{P}(J^{n}\setminus J^{n+1})\leqslant \mathbb{E}(R^{n}\mathbf{1}_{J^{n-1}})-\mathbb{E}(R^{n+1}\mathbf{1}_{J^{n}})+R^*\alpha(n+1)^2, 
\end{equation}
where $R^0 \mathbf{1}_{J^{-1}}=0$, $l=\max_n l(n)$, 
$$
R^*=\sqrt{\epsilon}\sqrt{4\epsilon+2\sqrt{\epsilon}}ML^2d+\left(\frac{1}{2}+8\epsilon+4\sqrt{\epsilon}\right)\left[(2d+4)G^2+\frac{M^2l^4 d_4}{18}+G^2\right],
$$
 and $d_4=d(d+2)(d+4)(d+6)$. By telescoping \eqref{telescoping attration}, we get
$$
\epsilon\mathbb{P}\left(\bigcup_{i=0}^n (J^{i}\setminus J^{i+1})\right)=\epsilon\sum_{i=0}^n\mathbb{P}(J^{i}\setminus J^{i+1})\leqslant \sum_{n=0}^{\infty}R^*\alpha(n)^2.
$$
Then, for $E^\infty$ defined by \eqref{eq:boundness event mu}, we have 
$$
\begin{aligned}
\mathbb{P}(E^\infty)&=\inf_n \mathbb{P}(E^{n+1})\geqslant \inf_n \mathbb{P}(J^{n})=\inf_n \mathbb{P}\left(\bigcap_{i=0}^n (J^{i}\setminus J^{i+1})^c\right)\\
&=\inf_n \mathbb{P}\left(\Omega\setminus\bigcup_{i=0}^n (J^{i}\setminus J^{i+1})\right)\geqslant1 -
{R^*\sum_{n=0}^\infty \frac{\alpha(n)^2}{\epsilon}}
.
\end{aligned}
$$
Thus, if $\sum_{n=0}^\infty \alpha(n)^2$ is small enough such that 
$1-R^*\sum_{n=0}^\infty \frac{\alpha(n)^2}{\epsilon}
>0$ holds, and the initial distance is not too large (i.e., $\x(0)\in U_0$), then we can conclude that $E^\infty$ occurs with a positive probability.
\end{proof}

\begin{proof}[Proof of \Cref{decay ln and an}]
From \Cref{attraction domain}, the boundedness event occurs with a positive probability. Define the events $E^n = \{ \x(i) \in U, i=1,2,\cdots,n \}$ with $U$ defined in \eqref{eq:boundness event mu}, corresponding indicator function $\mathbf{1}_{E^n}$, and $G=\max_{\x\in U} \Vert\nabla f(\x)\Vert_2$. Note that $l(n)\leqslant L\sqrt{\alpha(n)}$ (\Cref{assumption: decay difference length}), by \Cref{telescoping with inexact eigenvector}, we have that
$$
\begin{aligned}
&\mathbb{E}\left( \mathbf{1}_{E^{n+1}}\Vert \x(n+1) - \x^*\Vert_2^2 - \mathbf{1}_{E^n}\Vert\x(n) - \x^*\Vert_2^2 |\mathcal{F}(n) \right) \\
&\leqslant \mathbb{E}\left( \mathbf{1}_{E^{n}}\Vert \x(n+1) - \x^*\Vert_2^2 - \mathbf{1}_{E^n}\Vert\x(n) - \x^*\Vert_2^2 |\mathcal{F}(n) \right) \\
&\leqslant -\alpha(n)((1-\sqrt{\theta})\mu-M\theta-5M\sqrt{\theta})) \mathbf{1}_{E^n} \Vert\x(n) - \x^*\Vert_2^2 \\
&\qquad +2\alpha(n) \frac{Ml(n)^2d}{2} \frac{(1-\sqrt{\theta})\mu-M\theta-5M\sqrt{\theta}}{M}\\
&\qquad +\alpha(n)^2 \left((2d+4)G^2+{M^2l(n)^4 d(d+2)(d+4)(d+6)}/{18}\right)\\
&\leqslant -\alpha(n)((1-\sqrt{\theta})\mu-M\theta-5M\sqrt{\theta})) \mathbf{1}_{E^n} \Vert\x(n) - \x^*\Vert_2^2+M_4 \alpha(n)^2,
\end{aligned}
$$
where $M_4>0$ is a constant dependent on 
$d,\mu,M,G,\theta, \gamma,m,p,L$. By applying \Cref{lemma: An and Bn}, we conclude that $\mathbf{1}_{E^\infty}\Vert\x(n) - \x^*\Vert_2^2 \leqslant \mathbf{1}_{E^n} \Vert\x(n) - \x^*\Vert_2^2$ converges almost surely to zero. Consequently, conditioned on $E^\infty$, the \Cref{algorithm} reaches the $\epsilon_\x$-neighborhood of $\x^*$ in a finite number of iterations almost surely.
\end{proof}

\normalem
\bibliographystyle{siamplain}
\bibliography{references}

\end{document}